\documentclass[a4paper,11pt]{article}
\usepackage{amssymb,amsfonts,amsmath,amsthm,fullpage,graphicx,xcolor}

\newtheorem{theorem}{Theorem}[section]
\newtheorem{lemma}[theorem]{Lemma}
\newtheorem{propn}[theorem]{Proposition}
\newtheorem{cor}[theorem]{Corollary}
\theoremstyle{remark}
\newtheorem{remark}[theorem]{Remark}

\numberwithin{equation}{section}

\allowdisplaybreaks

\begin{document}

\title{Scaling limit for Brownian motions \\ on the $l$-level Sierpinski gaskets: \\ The fractal to Euclidean crossover}
\author{David A. Croydon\footnote{\tiny{Research Institute for Mathematical Sciences, Kyoto University, Kyoto 606--8502, Japan; croydon@kurims.kyoto-u.ac.jp.}}, Ben Hambly\footnote{\tiny{Mathematical Institute, University of Oxford, United Kingdom;  hambly@maths.ox.ac.uk.}}, Takashi Kumagai\footnote{\tiny{Department of Mathematics, Waseda University, 3-4-1 Okubo, Shinjuku-ku, Tokyo 169-8555, Japan; t-kumagai@waseda.jp.}}}
\maketitle

\begin{abstract}
In two dimensions, the $l$-level Sierpinski gasket $\mathrm{SG}(l)$ is obtained by splitting an equilateral triangle into a collection of $l^2$ equilateral triangles of equal size and with the same total area, retaining only the $l(l+1)/2$ triangles with the same orientation as the original triangle, and then iterating this procedure indefinitely. We show that the canonical diffusions on the spaces $\mathrm{SG}(l)$, $l\geq2$, can be rescaled to yield Brownian motion on the initial triangle. Our argument also applies to the analogous higher-dimensional Sierpinski gaskets. Moreover, we prove a local central limit theorem for the associated transition densities. Key to this is the derivation of a Poincar\'{e} inequality, in the proof of which we exploit the Euclidean-type mixing that occurs between the bottlenecks present at each scale of the fractal. 
%28A80 - Fractals
%60J25 - Continuous-time Markov processes on general state spaces
%60J35 - Transition functions, generators and resolvents
%60K50 - Anomalous diffusion models (subdiffusion, superdiffusion, continuous-time random walks, etc.)
\end{abstract}

\section{Introduction}

It is now well known that heat flow on fractal spaces, such as Sierpinski gaskets, has clear differences to that on classical Euclidean space. Indeed, unlike the diffusive behaviour of Brownian motion on the latter spaces, the canonical diffusions on many fractals have been shown to be sub-diffusive, with walk dimension strictly larger than 2. (We recall that the walk dimension is the polynomial exponent governing the space-time scaling of the process in question, see \eqref{walkdim} below for a precise definition in our setting.) In the physics literature, see \cite{MSS} in particular, the question was raised as to how the crossover from `fractal' to `Euclidean' behaviour occurs when one considers a sequence of fractals increasingly well-approximating Euclidean space. Establishing rigourously a result conjectured in \cite{MSS}, in \cite{HK}, this question was studied at the level of certain exponents describing the laws of the related stochastic processes. It was established that these exponents, which relate to the short-time, microscopic behaviour of the relevant diffusions, retain the fractal aspects of the model, even asymptotically. (We expand on this discussion in Theorem \ref{fobw} and Remarks \ref{HKcomp} and \ref{rem2} below.) In this article, we consider the same issue, but at the level of a functional scaling limit. In contrast to the story told in \cite{HK} and \cite{MSS}, we demonstrate that for a certain sequence of fractal spaces and a certain scaling regime, the fractality of the spaces is no longer preserved in the limit. Rather, we see that the associated sub-diffusive processes have as a scaling limit a standard Brownian motion on Euclidean space (restricted to lie within a suitable domain). Additionally, we present a corresponding result for the transition densities of the processes in question.

To describe the results in \cite{HK} and our main results, let us start by introducing some key notation; full definitions are postponed to Section \ref{prelimsec}. We will work in $d$-dimensional Euclidean space, $d\geq 2$, but usually suppress $d$ from the notation as this parameter does not vary. First, let $V_0$ be the vertices of a $d$-dimensional regular tetrahedron $\Delta$ of side length 1, one of the vertices of which is $0$. For each integer $l\geq 2$, we set $\mathrm{SG}(l)$ to be the $d$-dimensional $l$-level Sierpinski gasket with boundary $V_0$ (see Figure \ref{sglfig}), and suppose that $X^l=(X^l_t)_{t\geq0}$ is the canonical diffusion, i.e.\ Brownian motion, on this space, started from 0; the time-scaling of $X^l$ will be fixed so that the expected commute time between 0 and each other vertex of $\Delta$ is equal to 1. For each $l,d\geq 2$, the following limit, called the walk dimension, is known to exist:
\begin{equation}\label{walkdim}
d_w(l):=\left(\lim_{t\rightarrow0}\frac{\log \mathbf{E}\left(\left|X^l_t\right|^2\right)}{2\log t}\right)^{-1};
\end{equation}
moreover, it has been established that $d_w(l)\in (2,\infty)$, see \cite[Proposition 5.3]{KajMur}, for example. Additionally, each of the processes $X^l$ admits a transition density (or heat kernel) with respect to the canonical flat probability measure on the underlying space, $(p^l_t(x,y))_{x,y\in \mathrm{SG}(l),\:t>0}$ say. The spectral dimension is defined to be the exponent governing the asymptotic scaling in the eigenvalues and it typically determines
the behaviour of the on-diagonal part of the heat kernel at short times. In particular, for our purposes we can define
\[d_s(l):=\lim_{t\rightarrow0}\frac{2\log p^l_t(0,0)}{\log t}.\]
It is known that the spectral dimension exists for Brownian motion on $\mathrm{SG}(l)$. The following is the main result of \cite{HK}.

\begin{figure}[t]
  \centering
  \includegraphics[width=0.3\textwidth]{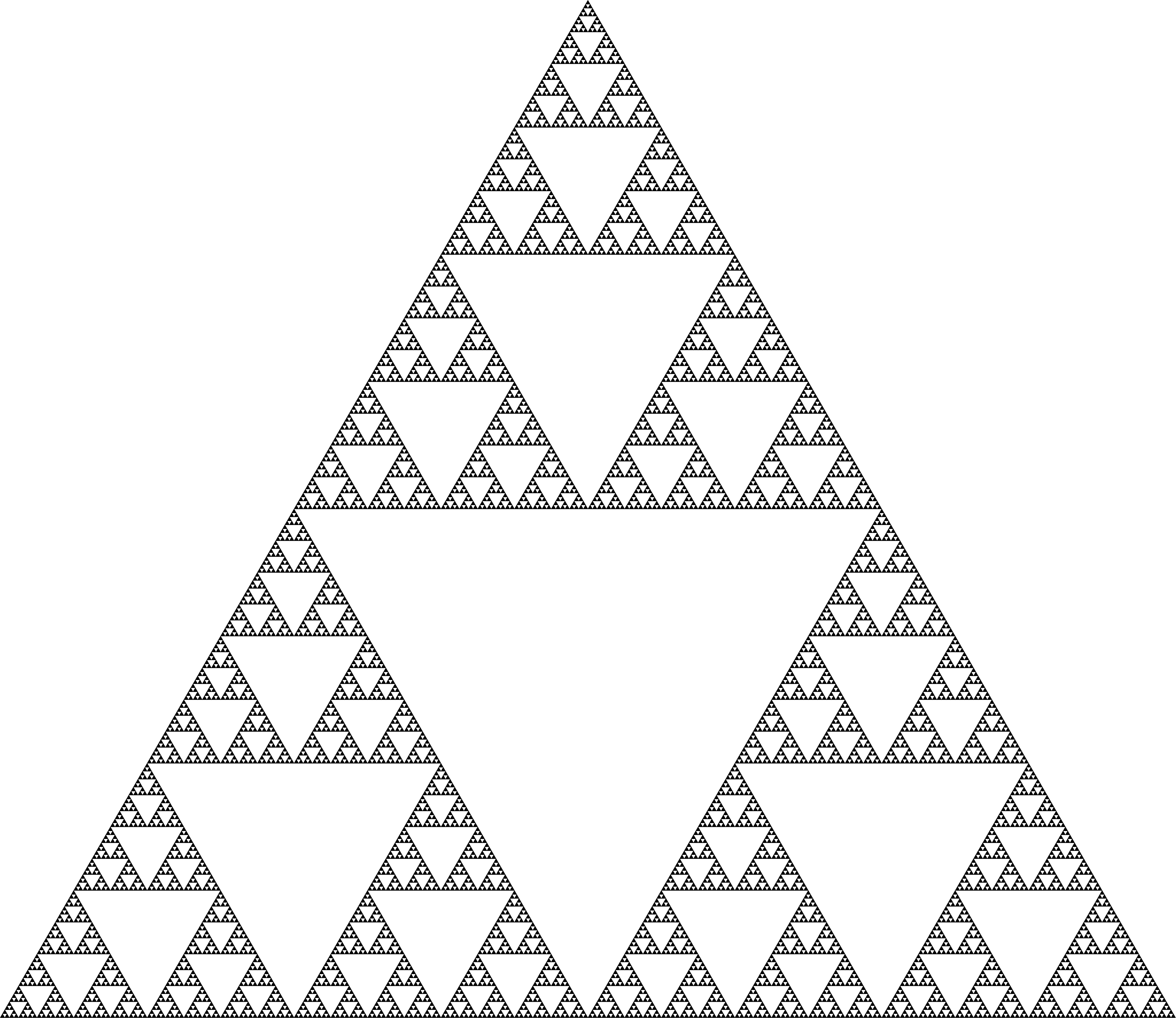}\hspace{5pt}\includegraphics[width=0.3\textwidth]{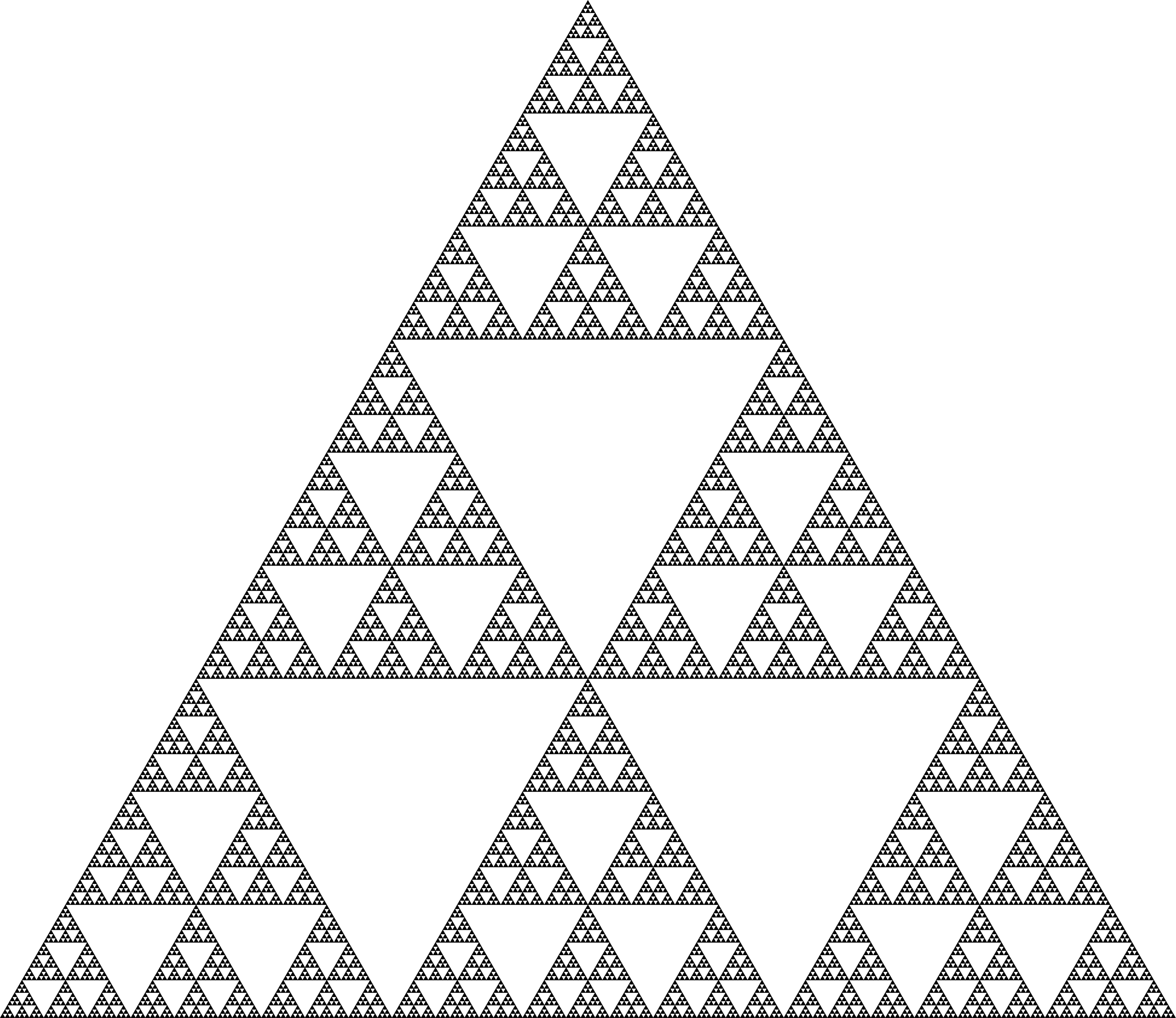}\hspace{5pt}\includegraphics[width=0.3\textwidth]{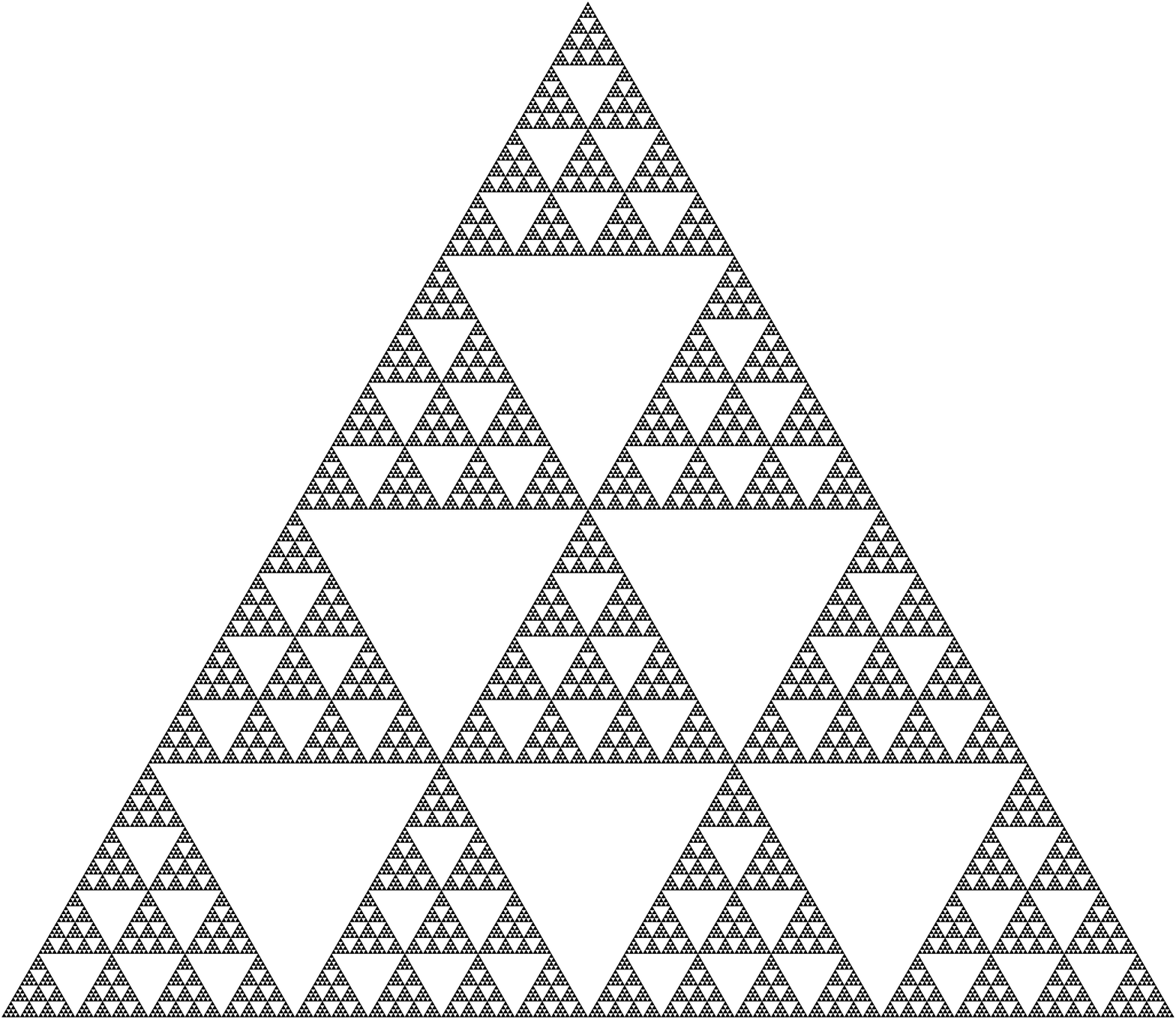}
  \caption{The sets $\mathrm{SG}(2)$, $\mathrm{SG}(3)$ and $\mathrm{SG}(4)$ in two dimensions.}\label{sglfig}
\end{figure}

\begin{theorem}{\rm (\cite[Corollary 2.3]{HK})}\label{fobw}
For each $d\geq 2$, the following hold as $l\to\infty$.
\begin{equation}\label{d_w_l}
d_w(l)=\left\{
           \begin{array}{ll}
             2+\frac{\log\log l}{\log l}+\frac{O(1)}{\log l}, & \hbox{if $d=2$;} \\
              d+\frac{O(1)}{\log l}, & \hbox{if $d\geq 3$.}
           \end{array}
         \right.\end{equation}
\begin{equation}\label{d_s_l}
d_s(l)=\left\{
           \begin{array}{ll}
             2-\frac{\log\log l}{\log l}+\frac{O(1)}{\log l}, & \hbox{if $d=2$;} \\
              2-\frac{O(1)}{\log l}, & \hbox{if $d\geq 3$.}
           \end{array}
         \right.\end{equation}
\end{theorem}

As this theorem shows, $d_w(l)\rightarrow d$ and $d_s(l)\rightarrow 2$ as $l\rightarrow \infty$, meaning the fractal aspects of the model are retained asymptotically for $d\ge 3$. In our first conclusion, we show that this is not the case for the sample paths of the processes when they are observed on an appropriate time scale. To present this result, we suppose that $X^\Delta=(X^\Delta_t)_{t\geq0}$ is standard Brownian motion on $\Delta$, reflected at the boundary. Unless otherwise stated, we assume that all stochastic processes start from 0. (This is not an essential requirement, but eases the presentation.)

\begin{theorem}\label{t:mr1} There exists a sequence of scaling parameters $(\tau^l)_{l\geq 2}$ taking values in $(0,\infty)$ such that, as $l\rightarrow\infty$,
\[\left(X^l_{\tau^lt}\right)_{t\geq 0}\rightarrow \left(X^\Delta_t\right)_{t\geq 0}\]
in distribution with respect to the local uniform topology on $C([0,\infty),\Delta)$, the space of continuous functions on $\Delta$ (equipped with the Euclidean distance).
\end{theorem}

To place this result into context, we note there have been substantial developments in the literature for understanding the scaling limits of processes on fractal-like spaces as described by so-called `resistance forms'. See \cite{Croy, CHK, Kig}, for example. However, we can not immediately apply the results in those works to deduce Theorem \ref{t:mr1} since, whilst the processes $X^l$, $l\geq 2$, are associated with resistance forms, the limiting process $X^\Delta$ is not. Indeed, the resistance metric (see Subsection \ref{ressec} for a definition) degenerates in the limit, and so fails to capture the behaviour of the process $X^\Delta$.

The scaling factors $(\tau^l)_{l\geq 2}$ of Theorem \ref{t:mr1} will initially be defined in terms of a certain hitting time, see \eqref{tauldef} below. Combining our definition of $\tau^l$ with \eqref{d_w_l}, we have the following result. (In the following, we write $f(l)\asymp g(l)$ if there exist $c_1,c_2>0$ such that $c_1g(l)\le f(l)\le c_2g(l)$ for all $l$, and $f(l)\sim g(l)$ if $\lim_{l\to \infty}f(l)/g(l)$ exists.)

\begin{theorem}\label{t:mr2} It holds that
\[\tau^l= \frac{1}{2}l^{2-d_w(l)}\asymp\left\{
                \begin{array}{ll}
                  \frac{1}{\log l}, & \hbox{if $d=2$;} \\
                  l^{2-d}, & \hbox{if $d\geq 3$.}
                \end{array}
              \right.\]
\end{theorem}

In particular, Theorem \ref{t:mr2} shows that $\tau^l\rightarrow0$ as $l\rightarrow\infty$, and so we are slowing down the diffusions $X^l$ by an increasing amount as $l\rightarrow\infty$. Intuitively this makes sense as the diffusions $X^l$ hit the vertices of $\Delta$ in finite time, whereas the process $X^\Delta$ does not.

\begin{remark}\label{HKcomp}
\begin{enumerate}
\item[(a)] The $\asymp$ in the statement of Theorem \ref{t:mr2} is a direct consequence of
\eqref{d_w_l}. It would be interesting to determine whether $\asymp$ could be replaced by $\sim$; this would require showing that the $O(1)$ terms in \eqref{d_w_l} are of the form $C+o(1)$ for some constant $C\in\mathbb{R}$. (See Remark \ref{resrem} below for further discussion of this point.)
\item[(b)] We highlight that the walk dimension of the limiting process $X^\Delta$ is 2, independent of the dimension. Hence, apart from in dimension two, the walk dimensions $d_w(l)$ do not converge to the walk dimension of $X^\Delta$ as $l\rightarrow\infty$. In particular, the fractality of the spaces persists in the walk dimension, but not at the level of processes on the time scale of Theorem \ref{t:mr1}. This difference can be explained by noting that the walk dimension is given by the arbitrarily short-time asymptotics of the diffusions $X^l$, whereas the scaling limit of Theorem \ref{t:mr1} depends of the behaviour of $X^l$ at a specified (albeit still short) time scale, as is required to enable the process to suitably explore the space.
\item[(c)] One might also consider the behaviour of $X^l$ on the infinite versions of the fractals, as given by $\mathrm{SG}_\infty(l):=\cup_{n=1}^\infty l^n\mathrm{SG}(l)$. In this case, the limiting space is naturally given by an infinite collection of tetrahedra only connected at vertices, see Figure \ref{sglinffig} for the two-dimensional case. On the time-scales considered in Theorem \ref{t:mr1}, the natural diffusions on $\mathrm{SG}_\infty(l)$, $X^{l,\infty}$ say, will take an increasingly long time to leave the initial tetrahedron. Indeed, the expected time it takes $X^{l,\infty}_{\tau^l\cdot}$ to hit a vertex in $V_0$ other than its starting location is $1/2\tau^l$. Since $1/\tau^l$ diverges, it is easy to see that the limit process must be trapped in the initial tetrahedron forever and, in fact, one readily sees that $X^{l,\infty}_{\tau^l\cdot}\rightarrow X^\Delta$ by the same argument as Theorem \ref{t:mr1}. This situation is to be expected given that, as already noted, the limiting Brownian motion $X^\Delta$ never hits points. On the other hand, by considering the process $X^{l,\infty}$ at hitting times of tetrahedral vertices in the limiting space, it is an elementary exercise to check that, for any sequence $\varepsilon_l\rightarrow0$,
    \[\left(\varepsilon_l X^{l,\infty}_{\varepsilon_l^{-2}t}\right)_{t\geq 0}\rightarrow \left(X^\angle_t\right)_{t\geq 0}\]
    in distribution, where $X^\angle$ is the Brownian motion on the cone given by the convex hull of $\mathrm{SG}_\infty(l)$ (for any $l$), reflected at the boundary. Finally, if one considers the processes $X^{l,\infty}$ without time-rescaling, we expect that one will see the processes mix within each tetrahedron in asymptotically zero time, and then, after an asymptotically exponential amount of time, jump to an adjacent tetrahedron, uniformly amongst the possible neighbours. (One could regard the process as a continuous-time Markov chain on the graph representing the part of the tetrahedral lattice obtained when one considers each tetrahedron as a point.) In particular, we conjecture that a functional scaling limit with respect to the uniform topology will not hold. However, it should be possible to describe the finite-dimensional distributions of the limiting process in terms of a continuous time random walk on tetrahedra in such a way that, conditional on where this walk is, the limiting process is uniformly distributed on the relevant tetrahedron. We leave checking the details of such a result as an open problem.
\end{enumerate}
\end{remark}

\begin{figure}[t]
  \centering
  \includegraphics[width=0.9\textwidth]{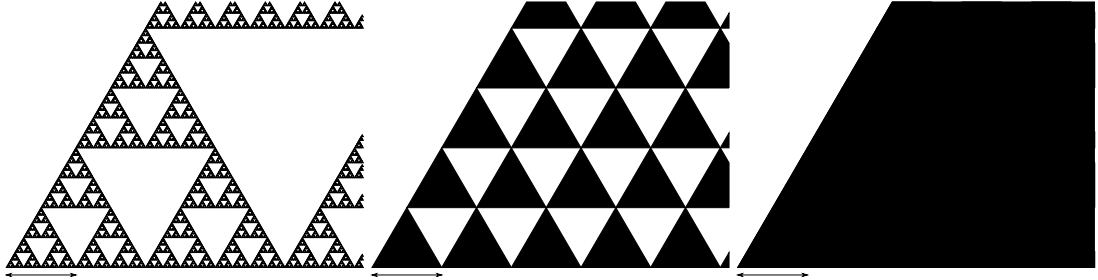}
  \caption{The set $\mathrm{SG}_\infty(2)$, the (Hausdorff) limit of the sets $\mathrm{SG}_\infty(l)$ (as $l\rightarrow\infty$) and the convex hull of $\mathrm{SG}_\infty(l)$ (for any $l$). The marked line segment shows the distance between vertices in $V_0$.}\label{sglinffig}
\end{figure}

As already advertised, in addition to a functional scaling limit, we can also give a scaling limit for the transition densities of $X^l$, $l\geq 2$. Recall that $(p^l_t(x,y))_{x,y\in \mathrm{SG}(l),\:t>0}$ is the transition density of Brownian motion on $\mathrm{SG}(l)$. Similarly, we write $(p^\Delta_t(x,y))_{x,y\in \Delta,\:t>0}$ for the transition density of $X^\Delta$ with respect the $d$-dimensional Hausdorff measure on $\Delta$ normalised to be a probability measure. We are then able to show the following uniform convergence.

\begin{theorem}\label{LCLT}
For every compact interval $I\subseteq (0,\infty)$, it holds that
\[\lim_{l\rightarrow\infty}\sup_{x,y\in \mathrm{SG}(l)}\sup_{t\in I}\left|p^l_{\tau^l t}(x,y)-p^\Delta_{t}(x,y)\right|=0.\]
\end{theorem}

\begin{remark}\label{rem2}
\begin{enumerate}
\item[(a)] The proof of Theorem \ref{LCLT} highlights the point at which the transition density evolves from the fractal scaling to the Euclidean. Indeed, in Lemma \ref{hkbound} below, we show that, for $t\leq l^{-d_w(l)}$, the on-diagonal supremum $\sup_{x\in\mathrm{SG}(l)}p^l_t(x,x)$ is bounded above by an expression of the form $t^{-d_f(l)/d_w(l)}$ (multiplied by a suitable $l$-dependent constant), whereas for $t\geq l^{-d_w(l)}$, the bound is of the form $t^{-d/2}$. Whilst we only need the upper bound, we expect that the nature of the $t$ dependence is precise (though possibly there will be a different $l$-dependent constant needed in a corresponding lower bound). This is natural to expect, since $l^{-d_w(l)}$ is precisely the time-scale on which the diffusion $X^l$ leaves a level 1 cell, and thus moves beyond the fractal scale to the more Euclidean structure in which these cells are arranged.
\item[(b)] A similar remark to Remark \ref{HKcomp}(b) could be made concerning the spectral dimensions $d_s(l)$ of the diffusions $X^l$. As noted in \eqref{d_s_l}, $d_s(l)\rightarrow 2$ as $l\rightarrow \infty$. Since the limiting process $X^\Delta$ has spectral dimension $d$, there is only continuity in these exponents in dimension 2. Again, comparing the results of \cite{HK} and this paper, we see that the fractality of the spaces is preserved on arbitrarily short time scales, but not on the time scales given by $\tau^l$ (which is much longer than the time-scale $l^{-d_w(l)}$ identified in the previous comment). In fact, in the arguments used to prove Theorem \ref{LCLT}, we are required to consider the interplay between the fractal and Euclidean nature of the spaces not just at the first level, but between each scale level along the sequence $l^{-m}$, $m\geq 0$.
  \item[(c)] Whilst we do not pursue it here, we believe that by appealing to the general framework of \cite{KS}, it should also be possible to show convergence of the rescaled eigenvalues and eigenfunctions of the generators of the processes $X^l$, $l\geq 2$, to those associated with the limiting process $X^\Delta$.
\end{enumerate}
\end{remark}

The remainder of the article is organised as follows. In Section \ref{prelimsec}, we present some basic background regarding the construction of the sets $\mathrm{SG}(l)$ and the associated processes $X^l$, as well as the limiting process $X^\Delta$. Section \ref{sec3} contains the proof of Theorem \ref{t:mr1}, a key input to which is an estimate on the diameter of $\mathrm{SG}(l)$ with respect to an effective resistance metric. Using some symmetries of the underlying objects, the proof of Theorem \ref{t:mr2} is also set out in this section. Next, in Section \ref{sec4}, we show the equicontinuity of the rescaled transition densities of the processes $X^l$, and as a consequence obtain Theorem \ref{LCLT}. Finally, in Section \ref{sec5}, we discuss some further examples to which the approach of this article should be applicable, and some for which additional work seems to be required. Unless specifically noted, constants such as $c, C,\dots$ may change from line to line.

\section{Preliminaries}\label{prelimsec}

In this section, we introduce some of the key notation that will be used throughout the rest of the article. As already noted, we will typically suppress the dimensional dependence from the notation.

\subsection{The $l$-level Sierpinski gaskets}

To construct an $l$-level Sierpinski gasket in $d$ dimensions, we start by considering a $d$-dimensional regular tetrahedron $\Delta$ of side length 1, with one vertex equal to 0. (For each dimension $d$, the set $\Delta$ will be fixed for all $l$.) For each $l\in\mathbb{N}$, let $N_{d,l}$ be the maximal number of translations of $l^{-1}\Delta$ that can be placed inside $\Delta$ and only intersect at vertices. There is only one arrangement that achieves this; we denote the corresponding translations of $l^{-1}\Delta$ by $(\Delta_j)_{j=1}^{N_{d,l}}$. For example, when $d=2$, these are the collection of upward pointing triangles that appear when $\Delta$ is tiled by triangles a factor $l$ smaller. (A similar description can also be given in higher dimensions.) As noted in \cite{HK}, it holds that $N_{2,l}=l(l+1)/2$, and $N_{d,l}$ can be computed for $d\geq 3$ via the formula
\[N_{d,l}=\sum_{k=1}^lN_{d-1,k}.\]
It is a simple exercise in induction (on $d$) to check that this implies
\begin{equation}\label{ndl}
N_{d,l}=\binom{d+l-1}{d}.
\end{equation}
Now, for each $j=1,\dots,N_{d,l}$, let $\psi_j$ be the orientation preserving similitude from $\Delta$ to $\Delta_j$. For $l\geq 2$, the $l$-level Sierpinski gasket that we will consider, $\mathrm{SG}(l)$, is then the unique compact set $K$ satisfying
\[K=\bigcup_{j=1}^{N_{d,l}}\psi_j(K).\]
See Figure \ref{sglfig} for illustrations of $\mathrm{SG}(2)$, $\mathrm{SG}(3)$ and $\mathrm{SG}(4)$ in two dimensions.

The set $\mathrm{SG}(l)$ is readily computed to have Hausdorff dimension given by
\[d_H\left(\mathrm{SG}(l)\right)=\frac{\log N_{d,l}}{\log l}.\]
Moreover, the corresponding $d_H(\mathrm{SG}(l))$-dimensional Hausdorff measure on $\mathrm{SG}(l)$ is a non-trivial, finite measure of full support. We will denote by $\mu^l$ the version of this measure that has been normalized to have total mass equal to one. In particular, it holds that $\mu^l(\mathrm{SG}(l)\cap \Delta_j)=1/N_{d,l}$ for each $j=1,\dots,N_{d,l}$.

In our subsequent arguments, it will be convenient to label the `vertices' of each $\mathrm{SG}(l)$. As noted in the introduction, we will denote by $V_0$ the vertices of the original regular tetrahedron $\Delta$. We will further define the $i$-th level vertices inductively by setting
\[V_{i+1}^l:=\bigcup_{j=1}^{N_{d,l}}\psi_j\left(V_i^l\right).\]
Note in particular that if we consider $V_1^l$ to be the vertices of a graph, with edges between nearest-neighbour vertices, then $V_1^l$ is simply the restriction of a tetrahedral lattice with edge length equal to $l^{-1}$ to those vertices within $\Delta$.

\subsection{Resistance and Brownian motion}\label{ressec}

To construct our Brownian motion $X^l=(X^l_{t})_{t\geq 0}$ on $\mathrm{SG}(l)$, we will apply the general machinery of \cite{Kig} (see also \cite{Kigq}), which covers a much wider collection of self-similar fractals. (We note that Brownian motion was constructed and detailed heat kernel
estimates were obtained earlier on nested fractals, which is a class of fractals including $\mathrm{SG}(l)$. See \cite{DobKus,FHK,Fuk92,HKpcf,Kum93,Lin}, for instance.) We start by considering the bilinear form on $V_0$ given by setting:
\[\mathcal{E}_0(f,f):=\frac{1}{d+1}\sum_{x,y\in V_0}\left(f(x)-f(y)\right)^2.\]
In particular, the above expression represents the energy dissipated in the network created when wires of resistance $(d+1)/2$ are placed between each pair of vertices in $V_0$ and vertices are held at voltages according to $f$. The choice of normalisation is made so that the corresponding effective resistance, as defined by setting $R_0(x,x):=0$ and, for $x\neq y$,
\begin{equation}\label{rdef}
R_0(x,y):=\left(\inf\left\{\mathcal{E}_0(f,f):\:f(x)=0,\:f(y)=1\right\}\right)^{-1},
\end{equation}
is simply the discrete metric on $V_0$, i.e.\ $R_0(x,y)=\mathbf{1}_{\{x\neq y\}}$. For each $d,l\geq 2$, it is then possible to deduce the existence of a constant $\rho_l=\rho_{d,l}>1$ such that if a sequence of bilinear forms $(\mathcal{E}_i^l)_{i\geq 1}$ is defined on the sets $(V_i^l)_{i\geq 1}$, respectively, by setting
\begin{equation}\label{edef}
\mathcal{E}_{i+1}^l(f,f):=\rho_{l}\sum_{j=1}^{N_{d,l}}\mathcal{E}^l_i\left(f\circ \psi_j,f\circ \psi_j\right),
\end{equation}
where $\mathcal{E}^l_0:=\mathcal{E}_0$, then the corresponding effective resistances $R_i^l$, defined analogously to \eqref{rdef}, satisfy
\begin{equation}\label{consist}
R_{i}^l(x,y)=R_{i'}^l(x,y)
\end{equation}
for every $i>i'$ and $x,y\in V^l_{i'}$. This compatibility means it is possible to introduce a metric $R^l$ on $\cup_{i=1}^\infty V_i^l$ such that $R^l=R^l_i$ on each $V^l_i$ (and $R^l=R_0$ on $V_0$). By \cite[Theorem 3.3.4]{Kig}, one can use continuity to extend $R^l$ to a metric on the whole $\mathrm{SG}(l)$ that is equivalent to the Euclidean metric (see \cite[Remark 3.7]{FHK}). Moreover, as per \cite[Definition 2.3.2]{Kig}, $R^l$ is a resistance metric on $\mathrm{SG}(l)$ and, by \cite[Theorem 2.3.6]{Kig}, there is an associated resistance form $(\mathcal{E}^l,\mathcal{F}^l)$, which is a bilinear form on $\mathrm{SG}(l)$ satisfying
\begin{equation}\label{rldef}
R^l(x,y):=\left(\inf\left\{\mathcal{E}^l(f,f):\:f\in\mathcal{F}^l,\:f(x)=0,\:f(y)=1\right\}\right)^{-1},
\end{equation}
for each $x,y\in \mathrm{SG}(l)$ with $x\neq y$. Finally, we note that the results of \cite[Section 2.4]{Kig} show that $(\mathcal{E}^l,\mathcal{F}^l)$ can also be considered as a regular Dirichlet form on $L^2(\mathrm{SG}(l),\mu^l)$. Thus, by the classical correspondence (e.g.\ \cite[Theorem 7.2.1]{FOT}), there exists a related Hunt process $X^l$, which can be checked to be a diffusion process; this is our Brownian motion. Since points have positive capacity in our setting (by \cite[Theorem 9.9]{Kigq}), the law of this process is uniquely defined from any starting point. We recall from the introduction that, for simplicity of presentation, we will henceforth suppose that $X^l$ is started from 0 unless otherwise stated.

\subsection{Vertex hitting times and an associated random walk}\label{23sec}

At the heart of our arguments will be the discrete-time random walk obtained by observing $X^l$ at the hitting times of vertices in $V_1^l$. To introduce this, we first recall that processes associated with resistance forms (on compact spaces) satisfy the commute time identity (see \cite[equation (2.17)]{CHK}, for example). In particular, if $\sigma^l_x:=\inf\{t\geq 0:\:X_t^l=x\}$, then
\begin{equation}\label{commute}
\mathbf{E}\left(\sigma^l_y\:\vline\:X^l_0=x\right)+\mathbf{E}\left(\sigma^l_x\:\vline\:X^l_0=y\right)=R^l(x,y),\qquad \forall x,y\in \mathrm{SG}(l).
\end{equation}
Applying this in conjunction with the strong Markov property for $X^l$, we obtain that the stopping times $(T_n^l)_{n\geq 0}$ are all well-defined and finite random variables, where $T_0^l=0$ and
\[T_n^l:=\inf\left\{t>T_{n-1}^l:\:X^l_t\in V_1^l\backslash\left\{X^l_{T_{n-1}^l}\right\}\right\}.\]
Moreover, from the obvious symmetries of the fractal $\mathrm{SG}(l)$, we have that the discrete-time process $(Y^l_n)_{n\geq 0}$ given by
\[Y^l_n:=X^l_{T_{n}^l},\qquad \forall n\geq 0,\]
is the simple random walk on $V_1^l$ (considered as a graph with nearest-neighbour edges) started from 0.

As a final piece of notation for this preliminary section, we define
\begin{equation}\label{tauldef}
\tau^l:=dl^2\mathbf{E}(T_1^l).
\end{equation}
In Section \ref{sec4}, we will later check that this is equivalent to the description of $\tau^l$ given in Theorem \ref{t:mr2}. As for Theorem \ref{t:mr1}, we will use a law of large numbers for the sequence $(T_n^l)_{n\geq 0}$ to deduce that
\[T_n^l\approx n\mathbf{E}(T_1^l) =n\tau^ll^{-2}d^{-1}\]
(see Proposition \ref{tprop}), which then allows us to show
\begin{equation}\label{sketch}
X^l_{\tau^lt}\approx X^l_{T^l_{\lfloor dl^2t\rfloor }}=Y^l_{\lfloor dl^2t\rfloor}\rightarrow X^\Delta_t,
\end{equation}
where the convergence result is a simple consequence of the fact that $Y^l$ is merely a random walk on a section of a tetrahedral lattice.

\section{Proof of Theorems \ref{t:mr1} and \ref{t:mr2}}\label{sec3}

In this section, we will prove the functional scaling limit of Theorem \ref{t:mr1} and the description of the scaling factors given by Theorem \ref{t:mr2}. In order to make rigourous the sketch proof of Theorem \ref{t:mr1} presented at the end of previous section, we start by giving the scaling limit of the discrete-time random walk $Y^l$. The proof involves quite standard techniques, but we include it for completeness. (In the case of two dimensions, a much easier argument is possible using the symmetries of the triangle.)

\begin{propn}\label{p:yscaling} As $l\rightarrow \infty$,
\[\left(Y^l_{\lfloor dl^2t\rfloor}\right)_{t\geq 0}\rightarrow \left(X^\Delta_t\right)_{t\geq 0}\]
in distribution with respect to the usual Skorohod $J_1$-topology on $D([0,\infty),\Delta)$, the space of cadlag functions on $\Delta$ (equipped with the Euclidean distance).
\end{propn}
\begin{proof}
We will apply the general result for convergence to a reflected diffusion of \cite[Theorem 2.1]{CCK}, checking the conditions discussed at the end of \cite[Section 2]{CCK}. (Technically, we should apply this to the continuous time version of $Y^l$ with unit mean exponential holding times, but for the various estimates we require, there is no problem to switch between discrete and continuous time.) We will first show tightness of the sequence $(Y^l_{\lfloor dl^2t\rfloor})_{t\geq 0}$, $l\geq 1$, in the following sense. For every $\varepsilon,T>0$,
\begin{equation}\label{tightnessclaim}
\lim_{\delta>0}\limsup_{l\rightarrow\infty}\sup_{x\in V_1^l}\mathbf{P}\left(\sup_{\substack{0\leq s,t\leq T:\\|s-t|\leq \delta}}\left|Y^l_{\lfloor dl^2t\rfloor}-Y^l_{\lfloor dl^2s\rfloor}\right|>\varepsilon\:\vline\:Y^l_0=x\right)=0.
\end{equation}
To this end, it will be convenient to consider the random walk on the graph with vertex set $V=\cup_{l'\geq 0}l'V_1^{l'}$, with edges between nearest-neighbour vertices. Note that $V$ is simply a section of the tetrahedral lattice contained within a cone. Write $Y=(Y_n)_{n\geq 0}$ for the discrete-time simple random walk on this graph, and $\mu^V$ for the measure on $V$ that assigns each vertex $x\in V$ a mass equal to its graph degree. It is clear that the graph metric on $V$ is bounded above and below by constant multiples of the Euclidean metric. Moreover, writing $\bar{B}(x,r):=\{y:\:|y-x|\leq r\}$ for the closed Euclidean ball, then it holds that $\mu^V(\bar{B}(x,r))\asymp 1+r^d$ for $x\in V$ and $r\geq 0$. In particular, $\mu^V$ is a doubling measure, i.e.\ there exists a constant $C$ such that
\begin{equation}\label{vd}
\mu^V\left(\bar{B}(x,2r)\right)\leq C\mu^V\left(\bar{B}(x,r)\right),\qquad \forall x\in V,\:r\geq 0.
\end{equation}
Applying \cite[Lemma 3.29 and Lemma 3.31]{barbook}, we further deduce that, for every $x\in V$ and $r\geq 0$,
\[\frac{1}{\mu^V\left(\bar{B}(x,r)\right)}\sum_{y,z\in V\cap \bar{B}(x,r)}\left(f(y)-f(z)\right)^2\leq C R_{\bar{B}(x,r)}^{-2}\sum_{\substack{y,z\in V\cap \bar{B}(x,r):\\|y-z|=1}}\left(f(y)-f(z)\right)^2,\]
where $R_{\bar{B}(x,r)}$ is the relative isoperimetric constant for $\bar{B}(x,r)$ of \cite[Definition 3.21]{barbook}. From \cite[Theorem 3.24]{barbook}, we have that
\[R_{\bar{B}(x,r)}\geq \frac{\left|V\cap \bar{B}(x,r)\right|}{\kappa_{\bar{B}(x,r)}},\]
where $\kappa_{\bar{B}(x,r)}$ is given by
\[\kappa_{\bar{B}(x,r)}:=\min_{\Gamma}\max_{\substack{y,z\in V\cap \bar{B}(x,r):\\|y-z|=1}}\left|\left\{u,v\in V\cap \bar{B}(x,r):\: \{y,z\}\in \gamma_{u,v}\right\}\right|,\]
with the first infimum being taken over collections $\Gamma$ of nearest neighbour paths with vertices in $ V\cap \bar{B}(x,r)$ that contain, for each pair $u,v\in  V\cap \bar{B}(x,r)$ an element $\gamma_{u,v}$ connecting $u$ to $v$.
Adapting the argument of \cite[Theorem 3.25]{barbook} from the square to the tetrahedral lattice, it is an elementary exercise to check that
\[\kappa_{\bar{B}(x,r)}\leq C\left| V\cap \bar{B}(x,r)\right|r.\]
Hence, we see that $R_{\bar{B}(x,r)}\geq C r^{-1}$, which implies in turn that, for every $x\in V$ and $r\geq 0$,
\begin{equation}\label{pi}
\frac{1}{\mu^V\left(\bar{B}(x,r)\right)}\sum_{y,z\in V\cap \bar{B}(x,r)}\left(f(y)-f(z)\right)^2\leq C r^{2}\sum_{\substack{y,z\in V\cap \bar{B}(x,r):\\|y-z|=1}}\left(f(y)-f(z)\right)^2.
\end{equation}
Now, \eqref{pi} is the Poincar\'{e} inequality for the graph $V$. It is well-known that in conjunction with the volume doubling property of \eqref{vd} and the fact the graph $V$ has bounded degree, this implies that the random walk $Y$ admits two-sided Gaussian transition density estimates (see \cite[Theorem 1.7]{Delmotte}). Furthermore, two-sided Gaussian estimates imply the following exit time estimates: for every $x\in V$, $r\geq 1$,
\begin{equation}\label{exitupper}
\sup_{y\in V\cap \bar{B}(x,r)}\mathbf{E}\left(\bar{\sigma}^{Y}(x,r)\:\vline\:Y_0=y\right)\leq c_1r^2,
\end{equation}
\[\mathbf{E}\left(\bar{\sigma}^{Y}(x,r)\:\vline\:Y_0=x\right)\geq c_2r^2,\]
where $\bar{\sigma}^Y(x,r):=\inf\{n\geq 0:\:Y_n\not\in  \bar{B}(x,r)\}$ (see \cite[Theorem 3.1]{GT1} and \cite[Proposition 6.1]{GT2}). In turn, these yield the following estimate: for every $x\in V$, $r\geq 1$ and $t>0$,
\[\mathbf{P}\left(\bar{\sigma}^{Y}(x,r)\leq t\:\vline\:Y_0=x\right)\leq Ce^{-cr^2/t},\]
see the argument of \cite[Theorem 5.5]{BBg}, for example. Clearly, the latter bound implies: for every $x\in V$, $r\geq 1$ and $t>0$,
\[\mathbf{P}\left(\sup_{0\leq n\leq t}\left|Y_n-x\right|>r\:\vline\:Y_0=x\right)\leq Ce^{-cr^2/t}.\]
By appealing to the symmetry of the graphs $V_1^l$, it readily follows that: for $\varepsilon\in(0,\tfrac12)$, $T>0$,
\begin{align*}
\lefteqn{\limsup_{l\rightarrow\infty}\sup_{x\in V_1^l}\mathbf{P}\left(\sup_{\substack{0\leq s,t\leq T:\\|s-t|\leq \delta}}\left|Y^l_{\lfloor dl^2t\rfloor}-Y^l_{\lfloor dl^2s\rfloor}\right|>\varepsilon\:\vline\:Y^l_0=x\right)}\\
&\leq\limsup_{l\rightarrow\infty}\sup_{x\in V}T\delta^{-1}\mathbf{P}\left(\sup_{0\leq n\leq 3dl^2\delta}\left|Y_{n}-x\right|>\varepsilon l\:\vline\:Y_0=x\right)\leq CT\delta^{-1}e^{-c\varepsilon^2/\delta}.
\end{align*}
Since the final expression here converges to 0 as $\delta\rightarrow 0$, the claim in \eqref{tightnessclaim} follows. Appealing to standard criteria for tightness of measures on $D([0,\infty),\Delta)$ (see \cite[Chapter 16]{Kall}, for example), we deduce that \cite[Assumption 2.6]{CCK} holds.

For \cite[Assumption 2.3]{CCK}, we need to check an exit time bound and a H\"{o}lder continuity result for harmonic functions. The first of these is implied by \eqref{exitupper}. As for the second, we again appeal to \cite[Theorem 1]{GT1} to observe that the two-sided Gaussian estimates for the transition density of $Y$ imply an elliptic Harnack inequality, i.e.\ there exists a constant $C<\infty$ such that, for every $x\in V$, $r\geq 1$ and $u:V\rightarrow\mathbb{R}$ that is harmonic in $V\cap\bar{B}(x,2r)$ (with respect to $Y$),
\begin{equation}\label{graphehi}
\sup_{x\in V\cap\bar{B}(x,r) }u(x)\leq C\inf_{x\in V\cap\bar{B}(x,r) }u(x).
\end{equation}
From this, standard arguments imply the existence of a constant $C$ and $\gamma>0$ such that
\[\left|u(x)-u(y)\right|\leq C\left(\frac{|x-y|}{r}\right)^\gamma\left\|u\right\|_{\infty},\qquad \forall x,y\in V\cap\bar{B}(x,r),\]
for every $u:V\rightarrow\mathbb{R}$ that is harmonic in $V\cap\bar{B}(x,2r)$. (We give an example of the argument in the proof of Theorem \ref{T.R1} below.) By scaling, the corresponding estimate holds for functions $u:V_1^l\rightarrow \mathbb{R}$ that are harmonic with respect to $Y^l$, with the constant being uniform over different values of $l$. This confirms \cite[Assumption 2.3]{CCK} in the present setting.

Next, suppose $\tilde{V}$ are the vertices of the tetrahedral lattice obtained by extending $V$ in the obvious way, and let $\tilde{Y}=(\tilde{Y}_n)_{n\geq 0}$ be the associated nearest-neighbour discrete-time random walk, started from 0. It is elementary to check that
\[\left(\varepsilon\tilde{Y}_{\lfloor d\varepsilon^{-2}t\rfloor}\right)_{t\geq 0}\rightarrow \left(X_t\right)_{t\geq 0},\]
where the limit process is standard Brownian motion on $\mathbb{R}^d$, started from 0. From this, it is straightforward to check that, if $\bar{Y}^l$ is $Y^l$ killed on hitting $\partial\Delta$, $\bar{X}^\Delta$ is $X^\Delta$ killed on hitting $\partial\Delta$, and $x_l\in V_1^l$ satisfies $x_l\rightarrow x\in \Delta$, then
\[\mathbf{P}\left(\left(\bar{Y}^l_{\lfloor dl^2t\rfloor}\right)_{t\geq 0}\in \cdot\:\vline\:\bar{Y}^l_0=x_l\right)\rightarrow\mathbf{P}\left(\left(X^\Delta_t\right)_{t\geq 0}\in \cdot\:\vline\:\bar{X}^\Delta_0=x\right),\]
weakly as probability measures on $D([0,\infty),\Delta)$. This confirms condition (ii) of \cite[Theorem 2.1]{CCK}.

It remains to check that the invariant measures of $Y^l$, suitably normalised, converge to the invariant measure of $X^\Delta$, and also that the Dirichlet form of any potential limiting process is bounded above by that of $X^\Delta$. The first of these claims is elementary. As for the second, this is also straightforward to check using the fact that: there exists a constant $c\in (0,\infty)$ such that, for any $f\in C_c^2(\Delta)$,
\[l^{2-d}\sum_{\substack{x,y\in V_1^l:\\|x-y|=l^{-1}}}\left(f(x)-f(y)\right)^2\rightarrow c\int_\Delta\left|\nabla f(x)\right|^2dx,\]
cf.\ the argument of \cite[Proposition 3.16]{CCK}, in which the more challenging example of a percolation cluster was considered. As a consequence, we can conclude that all of the conditions of \cite[Theorem 2.1]{CCK} hold, and the result follows.
\end{proof}

To check the remaining approximation in \eqref{sketch}, it will be useful to suitably understand the effective resistance metrics $R^l$ and the scaling constants $\rho_l$. In this direction, we first recall a result from \cite{HK}.

\begin{lemma}[{\cite[Theorem 2.2]{HK}}]\label{hklem} If $d=2$, then $\rho_l \asymp \log l$. If $d\geq 3$, then $\rho_l \asymp 1$.
\end{lemma}

\begin{remark}\label{resrem}
It will follow from observations made in the proof of Theorem \ref{t:mr2} that
\[\tau^l=\frac{l^2}{2\rho_lN_{d,l}}.\]
From the expression for $N_{d,l}$ in \eqref{ndl}, it follows that
\[\tau^l\sim\frac{d!}{2}l^{2-d}\rho_l^{-1}.\]
Hence to replace $\asymp$ by $\sim$ in Theorem \ref{t:mr2}, as suggested might be an interesting problem in Remark \ref{HKcomp}(b), it would be enough to show that $\rho_l\sim c\log l$ when $d=2$ and that $\rho_l$ has a limit when $d\geq 3$.
\end{remark}

Lemma \ref{hklem} will be sufficient for our purposes when $d=2$. However, for $d\geq 3$, we need slightly more control on the $\rho_l$, and in particular need that they are uniformly bounded away from 1.

\begin{lemma}\label{rhobound} For all $d\geq 2$,
\[\inf_{l\geq 2}\rho_l>1.\]
\end{lemma}
\begin{proof} Since we know that $\rho_l>1$ for each fixed $l$, the result for $d=2$ immediately follows from Lemma \ref{hklem}. Thus we restrict to the case when $d\geq 3$. For each $l$, suppose $V_1^l$ is the vertex set of a graph, equipped with nearest-neighbour edges. See Figure \ref{vl1fig}. Moreover, suppose $(R_{V_1^l}(x,y))_{x,y\in V_1^l}$ is the effective resistance on $V_1^l$ when edges are equipped with unit resistors. By \eqref{edef} and \eqref{consist}, we have that, for any $x\in V_0\backslash\{0\}$,
\[R_{V_1^l}(0,x)=\frac{2\rho_l}{d+1}R_{1}^l(0,x)=\frac{2\rho_l}{d+1}R_0(0,x)=\frac{2\rho_l}{d+1},\]
and so to prove the result, we are required to find a suitable lower bound for $R_{V_1^l}(0,x)$. To this end, fix $x\in V_0\backslash\{0\}$ and suppose $\Delta^l_0$ and $\Delta^l_1$ are the vertices in $V_1^l$ at graph distance $\lfloor l/2 \rfloor$ from 0 and $x$, respectively. (Again, see Figure \ref{vl1fig}.) Now, by Rayleigh's monotonicity principle and the parallel law (see \cite[Chapter 9]{LPW}, for example),
\[R_{V_1^l}(0,x)\geq R_{V_1^l}\left(0,\Delta_0^l\right)+R_{V_1^l}\left(\Delta_1^l,x\right)=2 R_{V_1^l}\left(0,\Delta_0^l\right).\]
By shorting vertices at each fixed distance from 0, it further holds that
\[R_{V_1^l}\left(0,\Delta_0^l\right)\geq \sum_{i=1}^{\lfloor l/2 \rfloor}\frac{1}{dN_{d-1,{i}}}.\]
Applying \eqref{ndl}, this yields that
\[R_{V_1^l}\left(0,\Delta_0^l\right)\geq  \frac{(d-1)!}{d}\sum_{i=1}^{\lfloor l/2 \rfloor}\frac{1}{i(i+1)\dots(i+d-2)},\]
and clearly the right-hand side here converges as $l\rightarrow\infty$ to
\[ \frac{(d-1)!}{d}\sum_{i=1}^{\infty}\frac{1}{i(i+1)\dots(i+d-2)}=\frac{d-1}{d(d-2)}.\]
Combining the above observations, we deduce that
\[\liminf_{l\rightarrow\infty}\rho_l\geq \frac{(d+1)(d-1)}{d(d-2)}>1.\]
Since $\rho_l>1$ for each fixed $l$, this completes the proof.
\end{proof}

\begin{figure}[t]
  \centering
  \includegraphics[width=0.9\textwidth]{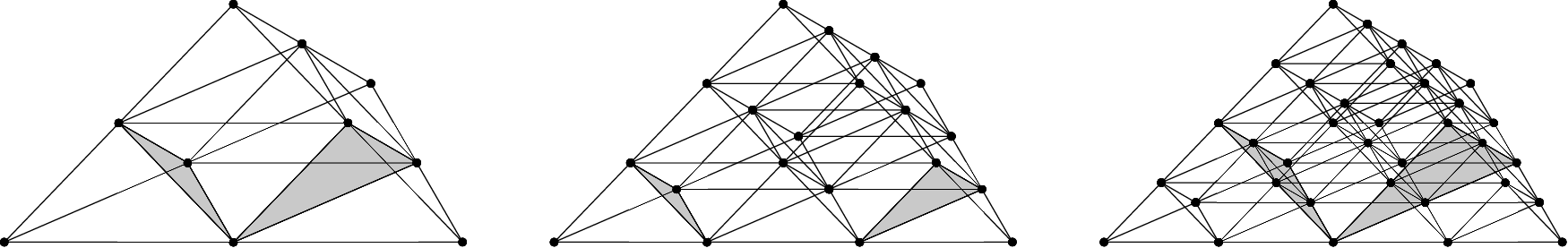}
  \caption{The sets $V^2_1$, $V^3_1$ and $V^4_1$ in three dimensions. In each case, the set $V_0$ consists of the four extremal vertices of the tetrahedron. The shaded regions show $\Delta^l_0$ and $\Delta^l_1$ when the resistance being estimated is that between the vertices in the bottom left and right of the figure.}\label{vl1fig}
\end{figure}

From the above estimates on $\rho_l$, we are able to deduce the following uniform bound on the diameters of the spaces $\mathrm{SG}(l)$ in terms of their effective resistance metrics. This result is fundamental for the arguments that follow.

\begin{lemma}\label{diambound} For all $d\geq 2$,
\[\sup_{l\geq 2}\sup_{x,y\in \mathrm{SG}(l)}R^l(x,y)<\infty.\]
\end{lemma}
\begin{proof} For a given $i\in \mathbb{N}$, we say that the sets of the form
\[\psi_{j_1}\circ\cdots\circ\psi_{j_i}(\Delta),\qquad j_1,\dots,j_i\in\{1,\dots,N_{d,l}\},\]
are $i$-cells; we also say that $\Delta$ is a 0-cell. Clearly, for each $x\in \mathrm{SG}(l)$, we can find a sequence $(x_i)_{i\geq 0}$ such that, for each $i\geq 0$, $x_i\in V_i^l$ (where $V_0^l:=V_0$), and $x$, $x_i$ and $x_{i+1}$ are in the same $i$-cell. It thus obviously holds that $x_i\rightarrow x$ with respect to the Euclidean metric, and, since this is topologically equivalent to the effective resistance metric, then $R^l(x_i,x)\rightarrow 0$. Supposing $(y_i)_{i\geq 0}$ is a similarly constructed sequence for $y\in \mathrm{SG}(l)$, it follows that
\begin{eqnarray*}
R^l(x,y)&=&\lim_{i\rightarrow\infty}R^l(x_i,y_i)\\
&\leq& \lim_{i\rightarrow\infty}\left(R^l(x_0,y_0)+\sum_{m=0}^{i-1}R^l(x_m,x_{m+1})+\sum_{m=0}^{i-1}R^l(y_m,y_{m+1})\right)\\
&\leq&1+2\sum_{m=0}^{\infty}\sup_{x',y'\in V_1^l}\sup_{j_1,\dots,j_m\in\{1,\dots,N_{d,l}\}}R^l\left(\psi_{j_1}\circ\cdots\circ\psi_{j_m}(x'),\psi_{j_1}\circ\cdots\circ\psi_{j_m}(y')\right).
\end{eqnarray*}
Now, by the self-similarity of the construction of the resistance form, i.e.\ \eqref{edef}, it holds that
\[R^l\left(\psi_{j_1}\circ\cdots\circ\psi_{j_m}(x'),\psi_{j_1}\circ\cdots\circ\psi_{j_m}(y')\right)\leq \rho_l^{-m}R^l\left(x',y'\right).\]
Hence, we have shown that
\[\sup_{x,y\in \mathrm{SG}(l)}R^l(x,y)\leq 1+2\sum_{m=0}^{\infty}\rho_l^{-m}\sup_{x',y'\in V_1^l}R^l\left(x',y'\right).\]
Consequently, since Lemma \ref{rhobound} implies that $\sum_{m=0}^{\infty}\rho_l^{-m}=(1-\rho_l^{-1})^{-1}$ is uniformly bounded over $l\geq 2$, to complete the proof it will suffice to show that the same is true of the supremum on the right-hand side above. Using the same notation as in the proof of Lemma \ref{rhobound}, we have that
\[\sup_{x',y'\in V_1^l}R^l\left(x',y'\right)=\frac{d+1}{2\rho_l}\sup_{x',y'\in V_1^l}R_{V_1^l}\left(x',y'\right).\]
And, by elementary arguments (cf.\ \cite[Proposition 9.16 and Exercise 9.1]{LPW}), one can check that
\[\sup_{x',y'\in V_1^l}R_{V_1^l}\left(x',y'\right)\leq \left\{
                                                         \begin{array}{ll}
                                                           C_d, & \hbox{for $d\geq 3$;} \\
                                                           C_2\log l, & \hbox{for $d=2$,}
                                                         \end{array}
                                                       \right.\]
where $C_d$ is a dimension dependent constant. Applying these estimates in conjunction with Lemma \ref{hklem}, we obtain the desired result.
\end{proof}

We are now able to show that the stopping times $T^l_n$, which were introduced in Subsection \ref{23sec}, concentrate on their means in the following manner. The constant $\tau^l$ was defined in \eqref{tauldef}.

\begin{propn}\label{tprop} For any $T>0$, as $l\rightarrow \infty$,
\[\frac{\sup_{n\leq l^2 T}\left|T_n^l-n\mathbf{E}(T_1^l)\right|}{\tau^l}\rightarrow 0\]
in probability.
\end{propn}

\begin{proof} Since $T_n^l-n\mathbf{E}(T_1^l)$ is a martingale (in $n$), we have from Doob's $L^2$ martingale inequality that
\[\mathbf{P}\left(\frac{\sup_{n\leq l^2 T}\left|T_n^l-n\mathbf{E}(T_1^l)\right|}{\tau^l}>\varepsilon\right)\leq \frac{\mathbf{E}\left(\left(T_{\lfloor l^2T\rfloor}^l-\lfloor l^2T\rfloor\mathbf{E}(T_1^l)\right)^2\right)}{(\varepsilon\tau^l)^2}\leq C_{d,\varepsilon,T}l^{-2}\frac{\mathbf{Var}(T_1^l)}{\mathbf{E}(T_1^l)^2},\]
and so it will be sufficient to show that the right-hand side converges to 0 as $l\rightarrow\infty$. For $A\subseteq \mathrm{SG}(l)$, we write
\[\sigma^l(A):=\inf\{t\geq 0:\:X^l_t\in A\}.\]
By the scaling of the resistance form $(\mathcal{E}^l,\mathcal{F}^l)$ that follows from \eqref{edef} and the scaling of the measure $\mu^l$, it is possible to check that, when $X^l$ is started from 0, $T_{1}^l$ is equal in distribution to $\rho_l^{-1}N_{d,l}^{-1}\sigma^l(V_{0}\backslash\{0\})$. Moreover, in the proof of Theorem \ref{t:mr2} below (see \eqref{req} in particular), we will check that $\mathbf{E}(\sigma^l(V_{0}\backslash\{0\}))=1/2d$. Hence it will be enough to show that
\[l^{-2}\mathbf{E}\left(\sigma^l(V_{0}\backslash\{0\})^2\right)\rightarrow 0\]
as $l\rightarrow\infty$. To do this, we will check that the expectation on the left-hand side above can be bounded uniformly in $l$. To this end, we first introduce a parameter
\[t_l:=\sup_{x\in \mathrm{SG}(l)}\mathbf{E}\left(\sigma^l(V_{0}\backslash\{0\})\:\vline\:X^l_0=x\right).\]
By the commute time identity (see \eqref{commute}), we have that
\[t_l\leq \sup_{x\in \mathrm{SG}(l)}\sup_{y\in V_{0}\backslash\{0\}}\mathbf{E}\left(\sigma^l_y\:\vline\:X^l_0=x\right)
\leq \sup_{x,y\in\mathrm{SG}(l)}R^l\left(x,y\right)\leq C,\]
where we have applied Lemma \ref{diambound} to deduce the second inequality with a constant $C$ that is independent of $l$. Consequently, applying the Markov property and the Markov inequality, we have that
\[\mathbf{P}\left(\sigma^l(V_{0}\backslash\{0\})\geq 2nt_l\right)\leq \sup_{x\in \mathrm{SG}(l)}\mathbf{P}\left(\sigma^l(V_{0}\backslash\{0\})\geq 2t_l\:\vline\:X^l_0=x\right)^n\leq 2^{-n},\qquad \forall n\in \mathbb{N}.\]
In particular, under $\mathbf{P}$, $\sigma^l(V_{0}\backslash\{0\})/2t_l$ is stochastically dominated by a geometric, parameter $1/2$, random variable, and so it has a second moment that is uniformly bounded in $l$. Since $t_l$ is also uniformly bounded in $l$, this completes the proof.
\end{proof}

Next, define $(H^l(t))_{t\geq 0}$, the right-continuous inverse of $T^l$, by setting
\[H^l(t):=\inf\left\{n\geq 0:\:T^l_n>t\right\}.\]
It is a straightforward consequence of the previous proposition (and the definition of $\tau^l$) that we have a similar concentration result for $H^l$.

\begin{cor}\label{c:hconc} For any $T>0$, as $l\rightarrow \infty$,
\[\sup_{t\leq T}\left|d^{-1}l^{-2}H^l(\tau^lt)- t\right|\rightarrow 0\]
in probability.
\end{cor}

Putting together Proposition \ref{p:yscaling} and Corollary \ref{c:hconc}, we arrive at the proof of Theorem \ref{t:mr1}.

\begin{proof}[Proof of Theorem \ref{t:mr1}]
Fix $\delta,T>0$, and suppose that
\[\sup_{t\leq T}\left|d^{-1}l^{-2}H^l(\tau^lt)- t\right|\leq \delta.\]
It then holds that
\begin{eqnarray*}
\sup_{t\leq T}\left|X^l_{\tau^lt}-Y^l_{\lfloor dl^2 t\rfloor}\right|&\leq &l^{-1}+\sup_{t\leq T}\left|Y^l_{H^l(\tau^lt)}-Y^l_{\lfloor dl^2 t\rfloor}\right|\\
&\leq &l^{-1}+\sup_{\substack{s,t\leq T+\delta:\\|s-t|\leq \delta}}\left|Y^l_{\lfloor dl^2 s\rfloor}-Y^l_{\lfloor dl^2 t\rfloor}\right|.
\end{eqnarray*}
Consequently, for any $\varepsilon,\delta>0$,
\begin{eqnarray*}
\lefteqn{\mathbf{P}\left(\sup_{t\leq T}\left|X^l_{\tau^lt}-Y^l_{\lfloor d l^2 t\rfloor}\right|> \varepsilon\right)}\\
&\leq &\mathbf{P}\left(\sup_{t\leq T}\left|d^{-1}l^{-2}H^l(\tau^lt)- t\right|> \delta\right)+\mathbf{P}\left(l^{-1}+\sup_{\substack{s,t\leq T+\delta:\\|s-t|\leq \delta}}\left|Y^l_{\lfloor d l^2 s\rfloor}-Y^l_{\lfloor dl^2 t\rfloor}\right|>\varepsilon\right).
\end{eqnarray*}
By Corollary \ref{c:hconc}, the first of these probabilities converges to 0 as $l\rightarrow \infty$. Moreover, Proposition \ref{p:yscaling} gives that the second probability converges to
\[\mathbf{P}\left(\sup_{\substack{s,t\leq T+\delta:\\|s-t|\leq \delta}}\left|X^\Delta_s-X^\Delta_t\right|>\varepsilon\right),\]
which, by the continuity of the process $X^l$, converges in turn to 0 as $\delta\rightarrow 0$. In particular, we have checked that, as $l\rightarrow\infty$,
\[\mathbf{P}\left(\sup_{t\leq T}\left|X^l_{\tau^lt}-Y^l_{\lfloor d l^2 t\rfloor}\right|> \varepsilon\right)\rightarrow 0,\]
which, in conjunction with Proposition \ref{p:yscaling}, is enough to complete the proof.
\end{proof}

To complete this section, we check the claims of Theorem \ref{t:mr2}, which reexpresses $\tau^l$, as originally defined in \eqref{tauldef}, in terms of the walk dimension, and describes the $l\rightarrow\infty$ asymptotics.

\begin{proof}[Proof of Theorem \ref{t:mr2}]
We start by recalling from the proof of Proposition \ref{tprop} that $T_1^l$ is equal in distribution to $\rho_l^{-1}N_{d,l}^{-1}\sigma^l(V_{0}\backslash\{0\})$, where $\sigma^l(V_{0}\backslash\{0\})$ is defined in the same proof. Hence, since $d_w(l)=\log(\rho_lN_{d,l})/\log l$ (by \cite[Proposition 2.1]{HK}, see also \cite[Theorem 8.18]{Barlow}),
\[\tau^l=dl^{2-d_w(l)}\mathbf{E}\left(\sigma^l(V_0\backslash\{0\})\right).\]
In particular, to complete the proof of the equality in the statement of the theorem, it suffices to show that
\begin{equation}\label{req}
\mathbf{E}\left(\sigma^l(V_0\backslash\{0\})\right)=\frac{1}{2d}.
\end{equation}
In checking this, we will apply the commute time identity of \eqref{commute}, which implies that, for any $x\in V_0\backslash\{0\}$,
\[\mathbf{E}\left(\sigma^l_x\right)=\frac{1}{2}R^l(0,x)=\frac{1}{2}.\]
Now, we may alternatively write
\begin{eqnarray*}
\lefteqn{\mathbf{E}\left(\sigma^l_x\right)}\\
&=&\mathbf{E}\left(\sigma^l_x\mathbf{1}_{\{\sigma^l_x=\sigma^l(V_0\backslash\{0\})\}}\right)+\mathbf{E}\left(\sigma^l_x\mathbf{1}_{\{\sigma^l_x>\sigma^l(V_0\backslash\{0\})\}}\right)\\
&=&\mathbf{E}\left(\sigma^l(V_0\backslash\{0\})\mathbf{1}_{\{\sigma^l_x=\sigma^l(V_0\backslash\{0\})\}}\right)+\\
&&\hspace{60pt}\mathbf{E}\left(\sigma^l(V_0\backslash\{0\})\mathbf{1}_{\{\sigma^l_x>\sigma^l(V_0\backslash\{0\})\}}\right)+\mathbf{E}\left(\sigma^l_x\right)\mathbf{P}\left(\sigma^l_x>\sigma^l(V_0\backslash\{0\})\right)\\
&=&\mathbf{E}\left(\sigma^l(V_0\backslash\{0\})\right)+\mathbf{E}\left(\sigma^l_x\right)\mathbf{P}\left(\sigma^l_x>\sigma^l(V_0\backslash\{0\})\right),
\end{eqnarray*}
where we have applied the strong Markov property at $\sigma^l(V_0\backslash\{0\})$ to deduce the second equality. (We also use symmetry to replace the expected time taken by the process $X^l$ to traverse from a vertex of $V_0\backslash\{0,x\}$ to $x$ with the expected time taken to traverse from $0$ to $x$.) Again appealing to symmetry, we have that
\[\mathbf{P}\left(\sigma^l_x>\sigma^l(V_0\backslash\{0\})\right)=\frac{|V_0|-2}{|V_0|-1}=\frac{d-1}{d}.\]
It follows that
\[\mathbf{E}\left(\sigma^l(V_0\backslash\{0\})\right)=\mathbf{E}\left(\sigma^l_x\right)\mathbf{P}\left(\sigma^l_x=\sigma^l(V_0\backslash\{0\})\right)=\frac{1}{2d},\] which establishes \eqref{req}, as required.

Finally, the $\asymp$ part of the theorem statement is a ready consequence of \eqref{d_w_l}.
\end{proof}

\section{Proof of Theorem \ref{LCLT}}\label{sec4}

In this section, we establish the convergence of transition densities stated as Theorem \ref{LCLT}. We first note that, by standard results on processes associated with resistance forms, for each $l\geq 2$, the process $X^l$ admits a jointly continuous transition density with respect to $\mu^l$ (see \cite[Theorem 10.4]{Kigq}, for example). As already noted in the introduction, we will denote this function by $(p^l_t(x,y))_{x,y\in \mathrm{SG}(l),\:t>0}$. Given the functional scaling limit of Theorem \ref{t:mr1}, to establish the convergence of these objects, it will suffice to check they satisfy the following equicontinuity result. Indeed, the proof of Theorem \ref{LCLT} can then be completed by a simple adaptation of the argument of \cite[Theorem 1]{CrHPA}. (In \cite{CrHPA}, the focus was on the transition densities of random walks on graphs, but the basic argument applies to more general processes.)

\begin{theorem}\label{PT6} For every compact interval $I\subseteq (0,\infty)$, it holds that
\[\lim_{\delta\rightarrow0}\limsup_{l\rightarrow \infty}\sup_{\substack{x,x',y,y' \in \mathrm{SG}(l):\\|x-x'|,|y-y'|\leq \delta}}\sup_{t\in I}\left|p^l_{\tau^l t}(x,y)-p^l_{\tau^l t}(x',y')\right|=0.\]
\end{theorem}

In fact, we will check something slightly stronger than the above statement, namely a uniform H\"{o}lder continuity property for the transition densities. To state this, we introduce the following scale function:
\[\Psi_l(r):= r^{d_w(l)}\mathbf{1}_{\{r\le l^{-1}\}}+2\tau^lr^2\mathbf{1}_{\{r> l^{-1}\}},\]
where, as per the characterization of Theorem \ref{t:mr2}, $\tau^l=\frac12l^{2-d_w(l)}$. We note that this function is a continuous, strictly increasing function, and so has a well-defined inverse $\Psi_l^{-1}$, which is given by
\[\Psi^{-1}_l(t):= t^{1/d_w(l)}\mathbf{1}_{\{t\le l^{-d_w(l)}\}}+(2\tau^l)^{-1/2}t^{1/2}\mathbf{1}_{\{t> l^{-d_w(l)}\}}.\]
Roughly speaking, $\Psi_l$ describes the time-scaling of $X^l$ on different scales: for $r\leq l^{-1}$, the appropriate time-scaling is the fractal one determined by $ r^{d_w(l)}$, whereas, for $r>l^{-1}$, the scaling is the usual Euclidean one $r^2$, scaled by the $l$-dependent constant $2\tau^l$. (The 2 here is simply to ensure that $\Psi_l$ and $\Psi^{-1}_l$ are continuous.) It is clear that Theorem \ref{PT6} is a consequence of the following statement. Note that we use the notation $x\vee y:= \max\{x,y\}$.

\begin{theorem}\label{PT7}
For any $t_0>0$, there exist constants $C$ and $\gamma>0$ such that
\[\left|p^{l}_{\tau^lt}(x,y)-p^{l}_{\tau^lt}(x',y')\right|\leq Ct^{-(1+3d/4)}\left(|x-x'|\vee|y-y'|\vee l^{-1}\right)^{\gamma},\]
for all $x,x',y,y' \in \mathrm{SG}(l)$, $t\in [4l^{-2},t_0]$.
\end{theorem}

Towards proving Theorem \ref{PT7}, we start by deriving an exit time estimate that establishes that $\Psi_l$ indeed gives appropriate control on the time-scaling, and also a version of the elliptic Harnack inequality in the present setting. For the statement, we let $B(x,r)$  be the standard $d$-dimensional Euclidean ball, centred at $x$ and with radius $r$, and define
\[\bar{\sigma}^l(x,r):=\sigma^l\left(B(x,r)^c\right)\]
to be the exit time of $B(x,r)$ by $X^l$. Moreover, we say that a function $h_l:\mathrm{SG}(l)\rightarrow \mathbb{R}$ is harmonic for $({\cal E}_{l}, {\cal F}_{l})$ in a ball $B(x_0,r)$ if $h_l(X^l_{t\land \bar{\sigma}^l(x_0,r)})$
is a martingale (with respect to the filtration associated with $X^l$) for every starting point $x\in \mathrm{SG}(l)$. To ease notation in what follows, we write $\mathbf{P}^l_x:=\mathbf{P}(\cdot\:\vline\:X^l_0=x)$ and $\mathbf{E}^l_x:=\mathbf{E}(\cdot\:\vline\:X^l_0=x)$.

\begin{lemma}\label{L.R2}
(a) There exists a constant $c_1$,
independent of $l$, such that
\[{\mathbf E}^l_x \bar{\sigma}^l(x_0,r)\leq c_1 \Psi_l(r), \qquad \forall x,x_0\in \mathrm{SG}(l),\:r\in(0,\tfrac14).\]
(b) There exists a constant $\theta>0$, independent of $l$, such that the following holds. If $r\in(8l^{-1},\frac{1}{4})$ and $h_l: \mathrm{SG}(l)\to {\mathbb R}_+$ is a harmonic function for $({\cal E}_{l}, {\cal F}_{l})$ on $B(x_0, 2r)$ for some $x_0\in \mathrm{SG}(l)$, then
\[\sup_{y\in B(x_0,r)\cap \mathrm{SG}(l)}h_l(y)\le \theta \inf_{y\in B(x_0,r)\cap \mathrm{SG}(l)}h_l(y).\]
\end{lemma}
\begin{proof}
We start with the proof of (a). Let $x,x_0\in \mathrm{SG}(l)$ and $r\in(0,\tfrac14)$. Started from $x$, if $X^l$ has exited $B(x,2r)$ at a certain time, then it will also have exited $B(x_0,r)$. In particular,
\[{\mathbf E}^l_x \bar{\sigma}^l(x_0,r)\leq {\mathbf E}^l_x \bar{\sigma}^l(x,2r).\]
Moreover, if $r\in (\tfrac14 l^{-1},\tfrac14)$ and the jump chain on $V_1^l$ has exited a ball of radius $l^{-1}(\lceil 2rl\rceil+1)$ centred at the first hitting location of $V_1^l$ by $X^l$, then $X^l$ will have exited $B(x,2r)$. Hence, by the definition of $\tau^l$,
\[{\mathbf E}^l_x \bar{\sigma}^l(x,2r)\leq \mathbf{E}^l_x\sigma^l(V_1^l)+\sup_{y\in V}\mathbf{E}\left(\bar{\sigma}^Y\left(y,\lceil 2rl\rceil+1\right)\:\vline\:Y_0=y\right)\mathbf{E}T_1^l,\]
where $\bar{\sigma}^Y$ was defined in the proof of Proposition \ref{p:yscaling}, and $\mathbf{E}T_1^l$ was defined in Subsection \ref{23sec}. Now, by scaling and the commute time identity, see \eqref{commute}, the first term above satisfies
\[ \mathbf{E}^l_x\sigma^l(V_1^l)\leq \rho_l^{-1}N_{d,l}^{-1}\sup_{y,z\in \mathrm{SG}(l)}R^l(z,y)\leq C\tau^l l^{-2}\leq C\tau^l r^2,\]
where the second inequality follows from the description of $\tau^l$ in Remark \ref{resrem} and the upper bound for the resistance diameter of Lemma \ref{diambound}, and the third inequality follows from the choice of $r$. Furthermore, recalling the definition of the lattice random walk $Y$ from the proof of Proposition \ref{p:yscaling}, it holds that
\[\sup_{y\in V}\mathbf{E}\left(\bar{\sigma}^Y\left(y,\lceil 2rl\rceil+1\right)\:\vline\:Y_0=y\right)\mathbf{E}T_1^l\leq C\tau^ll^{-2}\left(\lceil 2rl\rceil+1\right)^2\leq C\tau^l r^2=\frac12 C\Psi_l(r),\]
where to deduce the first inequality, we apply the definition of $\tau^l$ from \eqref{tauldef} and the bound on the exit times of $Y$ given in \eqref{exitupper}. Hence we have shown the desired bound for this range of $r$. Similarly, if $r\in (\tfrac14l^{-(m+1)},\tfrac14l^{-m}]$, then we have
\begin{align*}
{\mathbf E}^l_x \bar{\sigma}^l(x_0,r)&\leq {\mathbf E}^l_x \bar{\sigma}^l(x,2r)\\
&\leq  \mathbf{E}^l_x\sigma^l(V_{m+1}^l)+\rho_l^{-{m+1}}N_{d,l}^{-{m+1}}\sup_{y\in V}\mathbf{E}\left(\bar{\sigma}^Y\left(y,\lceil 2rl^{m+1}\rceil+1\right)\:\vline\:Y_0=y\right)\\
&\leq \rho_l^{-(m+1)}N_{d,l}^{-(m+1)}\left(\sup_{y,z\in \mathrm{SG}(l)}R^l(z,y)+\left(\lceil 2rl^{m+1}\rceil+1\right)^2\right)\\
&\leq Cl^{-d_w(l)(m+1)}\left(1+r^2l^{2(m+1)}\right)\\
&\leq Cr^{d_w(l)},
\end{align*}
where for the final inequality we use the fact that $d_w(l)\geq 2$. This establishes the exit time upper bound.

Let $r\in(8 l^{-1},\frac{1}{4})$. If $h_l$ is harmonic on $B(x_0,r)$ with respect to $X^l$, then it is harmonic on $\bar{B}(x_0,2r(1-l^{-1}))\cap V_1^l$ with respect to $Y^l$. Hence, similarly to \eqref{graphehi},
\[\sup_{x\in \bar{B}\left(x_0,\tfrac32r(1-l^{-1})\right)\cap V_1^l}h_l(x)\leq C\inf_{x\in \bar{B}\left(x_0,\tfrac32r(1-l^{-1})\right)\cap V_1^l}h_l(x).\]
Note that this statement follows directly from \eqref{graphehi} and the remark following \cite[Definition 1.44]{barbook}, which explains how there is no problem in incorporating the factor $\tfrac{3}{2}$ if one restricts the range of $r$ suitably; this is the reason for taking $r\geq 8l^{-1}$ rather than $r\geq l^{-1}$. Now, since harmonic functions take their maximum and minimum values on the boundaries of domains and the 1-cells of $\mathrm{SG}(l)$ only meet at vertices in $V_1^l$, the above inequality readily implies the desired result.
 \end{proof}

\begin{remark}
It is natural to ask whether the exit time upper bound can be complemented by a lower bound of the form
${\mathbf E}^l_{x_0} \bar{\sigma}^l(x_0,r)\geq c\Psi_l(r)$ for $x_0\in \mathrm{SG}(l)$, $r>0$. However, by arguing similarly to the previous proof, we are only able to show that, for $r\in (\tfrac14l^{-(m+1)},\tfrac14l^{-m}]$,
\[{\mathbf E}^l_{x_0} \bar{\sigma}^l(x_0,r)\geq cr^2l^{-(d_w(l)-2)(m+1)}\geq cr^{d_w(l)}l^{-(d_w(l)-2)}=cl^{-(d_w(l)-2)}\Psi_l(r).\]
In particular, there is additional $l$-dependence in the constant. Whilst we believe there should be some $l$-dependence as one passes along scales of the fractal from $l^{-m}$ to $l^{-(m+1)}$, we do not know if this constant is optimal.
\end{remark}

We next derive a version of the Poincar\'{e} inequality that is tailored to the different levels of the model with the scaling function $\Psi_l(r)$. Again, the proof follows the standard argument for fractals, but we have to combine this with an estimate for lattices in an appropriate way to handle the situation between the scales $l^{-m}$ and $l^{-(m+1)}$ for each $m\geq 0$. The restriction to the $m$-cell $\psi_i(\mathrm{SG}(l))$ seems rather awkward, but we believe that it is necessary, since without this, the comparison with the triangular lattice, which allows us to use the bound in \eqref{pi}, would not be applicable. Indeed, bottlenecks between $m$-cells will mean the result will not be true for more general sets of a similar size. Intuitively, mixing within an $m$-cell of the random walk happens at a much quicker rate than transitions between cells; this suggests functions like the heat kernel will, on the appropriate time-scale, be somewhat step-like -- smoothing on each $m$-cell, before mass passes to a neighbouring one. Indeed, the situation that arises between two neighbouring $m$-cells is, on the appropriate scale, somewhat analogous to the joining of two copies of $\mathbb Z^d$, $d\ge 2$, at a single point. It is known that the latter graphs violate the Poincar\'{e} inequality (see \cite[Example 3.32]{barbook} and \cite[Section 3.3]{Kum14}), and we find a similar issue in our setting. Thus we believe it is unavoidable that we have to consider each $m$-cell $\psi_i(\mathrm{SG}(l))$ separately.

\begin{lemma}
\label{thm:lem4.5}
There exists finite constants $C>0$ and $C_0>1$ such that for $l\geq 2$: if $r\in  (\tfrac14l^{-(m+1)},\tfrac14l^{-m}]$ for some $m\geq 0$, $x\in \psi_i(\Delta)$ for some $i=(i_1,\dots,i_{m})\in\{1,\dots,N_{d,l}\}^{m}$, where we define $\psi_i:=\psi_{i_1}\circ\cdots\circ \psi_{i_{m}}$, and $f\in\mathcal{F}^l$, then
\begin{equation}\label{poincare}
\int_{B(x,r)\cap\psi_i(\mathrm{SG}(l))}\left(f(y)-\bar{f}^{l,i}_{B(x,r)}\right)^2\mu^l(dy)\leq C\Psi_l(r)\rho_l^{m+1}
\sum_{\substack{j\in \{1,\dots,N_{d,l}\}^{m+1}:\\(j_1,\dots,j_m)=(i_1,\dots,i_m),\\\psi_j(\mathrm{SG}(l))\cap B(x,5r)\neq \emptyset}}\mathcal{E}^l\left(f\circ\psi_j,f\circ\psi_j\right),
\end{equation}
where
\[\bar{f}^{l,i}_{B(x,r)}:=\frac{1}{\mu^l(B(x,r)\cap\psi_i(\mathrm{SG}(l)))}\int_{B(x,r)\cap\psi_i(\mathrm{SG}(l))}f(y)\mu^l(dy).\]
\end{lemma}
\begin{proof} In the proof, we will write $B=B(x,r)\cap\psi_i(\mathrm{SG}(l))$, where $i\in\{1,\dots,N_{d,l}\}^{m}$ is such that $x\in \psi_i(\Delta)$, with $m$ satisfying $r\in (\tfrac14l^{-(m+1)},\tfrac14l^{-m}]$. Note that the left-hand side of \eqref{poincare} can be written
\begin{align*}
\lefteqn{\int_{B}\left(f(y)-\bar{f}^{l,i}_{B(x,r)}\right)^2\mu^l(dy)}\\
&=\frac{1}{2\mu^l(B)}\int_{B}\int_{B}\left(f(y)-f(z)\right)^2\mu^l(dy)\mu^l(dz)\\
&=\frac{1}{2\mu^l(B)}\sum_{j\in\{1,\dots,N_{d,l}\}^{m+1}}\sum_{k\in\{1,\dots,N_{d,l}\}^{m+1}}\int_{B\cap\psi_j(\mathrm{SG}(l))}\int_{B\cap\psi_k(\mathrm{SG}(l))}\left(f(y)-f(z)\right)^2\mu^l(dy)\mu^l(dz).
\end{align*}
We next replace $f(y)-f(z)$ by
\[f(y)-f(\psi_j(0))+f(\psi_j(0))-f(\psi_k(0))+f(\psi_k(0))-f(z)\]
and use that $(a+b+c)^2\leq 3(a^2+b^2+c^2)$ to deduce that
\begin{align}
\lefteqn{\int_{B\cap\psi_j(\mathrm{SG}(l))}\int_{B\cap\psi_k(\mathrm{SG}(l))}\left(f(y)-f(z)\right)^2\mu^l(dy)\mu^l(dz)}\nonumber\\
&\leq 3\int_{B\cap\psi_j(\mathrm{SG}(l))}\left(f(y)-f(\psi_j(0))\right)^2\mu^l(dy)\mu^l\left(B\cap\psi_k(\mathrm{SG}(l))\right)\nonumber\\
&\hspace{40pt}+3\left(f(\psi_j(0))-f(\psi_k(0))\right)^2\mu^l\left(B\cap\psi_j(\mathrm{SG}(l))\right)\mu^l\left(B\cap\psi_k(\mathrm{SG}(l))\right)\nonumber\\
&\hspace{40pt}+3\int_{B\cap\psi_k(\mathrm{SG}(l))}\left(f(z)-f(\psi_k(0))\right)^2\mu^l(dz)\mu^l\left(B\cap\psi_j(\mathrm{SG}(l))\right).\label{threeterms}
\end{align}
Now, to handle the first term here, we apply \eqref{rldef} and Lemma \ref{diambound} to obtain
\begin{align*}
{\int_{B\cap\psi_j(\mathrm{SG}(l))}\left(f(y)-f(\psi_j(0))\right)^2\mu^l(dy)}
&\leq N_{d,l}^{-(m+1)}\int_{\mathrm{SG}(l)}\left(f\circ\psi_j(y)-f\circ\psi_j(0)\right)^2\mu^l(dy)\\
&\leq N_{d,l}^{-(m+1)}\sup_{y\in \mathrm{SG}(l)}R^l(0,y)\mathcal{E}^l(f\circ\psi_j,f\circ\psi_j)\\
&\leq CN_{d,l}^{-(m+1)}\mathcal{E}^l(f\circ\psi_j,f\circ\psi_j).
\end{align*}
Obviously, it is also the case that, if $(j_1,\dots,j_m)\neq (i_1,\dots,i_m)$ or $B\cap\psi_j(\mathrm{SG}(l))=\emptyset$, then the left-hand side above is zero. Hence
\begin{align*}
\lefteqn{\frac{1}{\mu^l(B)}\sum_{j\in\{1,\dots,N_{d,l}\}^{m+1}}\sum_{k\in\{1,\dots,N_{d,l}\}^{m+1}}\int_{B\cap\psi_j(\mathrm{SG}(l))}\left(f(y)-f(\psi_j(0))\right)^2\mu^l(dy)\mu^l\left(B\cap\psi_k(\mathrm{SG}(l))\right)}\\
&\leq CN_{d,l}^{-(m+1)}\sum_{\substack{j\in\{1,\dots,N_{d,l}\}^{m+1}:\\(j_1,\dots,j_m)=(i_1,\dots,i_m),\\\psi_j(\mathrm{SG}(l))\cap B(x,r)\neq \emptyset}}\mathcal{E}^l(f\circ\psi_j,f\circ\psi_j).\hspace{200pt}
\end{align*}
The third term in \eqref{threeterms} can be dealt with in the same way. As for the second term, we have that
\begin{align*}
\lefteqn{\frac{1}{\mu^l(B)}\sum_{j,k\in\{1,\dots,N_{d,l}\}^{m+1}}\left(f(\psi_j(0))-f(\psi_k(0))\right)^2\mu^l\left(B\cap\psi_j(\mathrm{SG}(l))\right)\mu^l\left(B\cap\psi_k(\mathrm{SG}(l))\right)}\\
&\leq \frac{1}{N_{d,l}^{2(m+1)}\mu^l(B)}\sum_{\substack{j,k\in\{1,\dots,N_{d,l}\}^{m+1}:\\(j_1,\dots,j_m)=(i_1,\dots,i_m),\\(k_1,\dots,k_m)=(i_1,\dots,i_m)}}
\left(f(\psi_j(0))-f(\psi_k(0))\right)^2\mathbf{1}_{\{B\cap\psi_j(\mathrm{SG}(l))\neq\emptyset,\:B\cap\psi_k(\mathrm{SG}(l))\neq\emptyset\}}\\
&\leq\frac{1}{N_{d,l}^{2(m+1)}\mu^l(B)}\sum_{\substack{j,k\in\{1,\dots,N_{d,l}\}^{m+1}:\\(j_1,\dots,j_m)=(i_1,\dots,i_m),\\(k_1,\dots,k_m)=(i_1,\dots,i_m)}}
\left(f(\psi_j(0))-f(\psi_k(0))\right)^2\mathbf{1}_{\{\psi_j(0),\psi_k(0)\in B(x,r+l^{-(m+1)})\}}\\
&\leq\frac{1}{N_{d,l}^{2(m+1)}\mu^l(B)}\sum_{y,z\in V_{m+1}^l\cap B(x,r+l^{-(m+1)})\cap\psi_i(\mathrm{SG}(l))}\left(f(y)-f(z)\right)^2.
\end{align*}
Since the graph structure of $V_{m+1}^l\cap\psi_i(\mathrm{SG}(l))$ is simply a part of the triangular lattice, we can apply the discrete-space Poincar\'{e} inequality of \eqref{pi} to deduce that there exists a constant $C$ such that
\begin{align}
\lefteqn{\frac{1}{\mu^l(B)}\sum_{j,k\in\{1,\dots,N_{d,l}\}^{m+1}}\left(f(\psi_j(0))-f(\psi_k(0))\right)^2\mu^l\left(B\cap\psi_j(\mathrm{SG}(l))\right)\mu^l\left(B\cap\psi_k(\mathrm{SG}(l))\right)}\nonumber\\
&\leq  \frac{C}{N_{d,l}^{2(m+1)}\mu^l(B)}(rl^{m+1})^2\left| B_{i,m}\right|\sum_{\substack{y,z\in B_{i,m}:\\|y-z|=l^{-(m+1)}}}\left(f(y)-f(z)\right)^2\label{here},\hspace{100pt}
\end{align}
where we write $B_{i,m}:= V_{m+1}^l\cap B(x,r+l^{-(m+1)})\cap\psi_i(\mathrm{SG}(l))$.

Each of the sums can be bounded as follows:
\begin{align*}
\lefteqn{\sum_{\substack{y,z\in B_{i,m}:\\|y-z|=l^{-(m+1)}}}\left(f(y)-f(z)\right)^2}\\
&\leq (d+1)\sum_{\substack{j\in \{1,\dots,N_{d,l}\}^{m+1}:\\(j_1,\dots,j_m)=(i_1,\dots,i_m)}}\mathcal{E}_0\left(f\circ \psi_j,f\circ \psi_j\right)\mathbf{1}_{\{\psi_j(\mathrm{SG}(l))\cap  B_{i,m}\neq\emptyset\}}\\
&\leq (d+1)\sum_{\substack{j\in \{1,\dots,N_{d,l}\}^{m+1}:\\(j_1,\dots,j_m)=(i_1,\dots,i_m)}}\mathcal{E}^l\left(f\circ \psi_j,f\circ \psi_j\right)\mathbf{1}_{\{\psi_j(\mathrm{SG}(l))\cap B(x,5r)\neq \emptyset\}},
\end{align*}
where we have used that $\mathcal{E}_{0}(f\circ \psi_j,f\circ \psi_j)\leq \mathcal{E}^l(f\circ \psi_j,f\circ \psi_j)$ and $B(x,r+l^{-1})\subseteq B(x,5r)$ to obtain the second inequality. Moreover, it is elementary to check that $| B_{i,m}|\leq C N_{d,l}^{m+1}\mu^l(B)$. Hence, the expression in \eqref{here} is bounded above by
\[ \frac{C(rl^{m+1})^2}{N_{d,l}^{m+1}}\sum_{\substack{j\in \{1,\dots,N_{d,l}\}^{m+1}:\\(j_1,\dots,j_m)=(i_1,\dots,i_m)}}\mathcal{E}^l\left(f\circ \psi_j,f\circ \psi_j\right)\mathbf{1}_{\{\psi_j(\mathrm{SG}(l))\cap B(x,5r)\neq \emptyset\}}=\frac{C(rl^{m+1})^2}{\rho_l^{m+1}N_{d,l}^{m+1}}\Sigma,\]
where we use the abbreviation
\[\Sigma:=\rho_l^{m+1}
\sum_{\substack{j\in \{1,\dots,N_{d,l}\}^{m+1}:\\(j_1,\dots,j_m)=(i_1,\dots,i_m),\\\psi_j(\mathrm{SG}(l))\cap B(x,5r)\neq \emptyset}}\mathcal{E}^l\left(f\circ\psi_j,f\circ\psi_j\right).\]
Combining this estimate with the bounds on the first and third terms, we conclude that

\[\int_{B}\left(f(y)-\bar{f}^{l,i}_{B(x,r)}\right)^2\mu^l(dy)\leq C\rho_l^{-(m+1)}N_{d,l}^{-(m+1)}\left(1+(rl^{m+1})^2\right)\Sigma.\]
Writing all the scaling factors in terms of $l$ and applying the assumption that $r\geq\tfrac14 l^{-(m+1)}$, the right-hand side here is bounded above by
\[Cr^2l^{-(d_w(l)-2)(m+1)}\Sigma.\]
Finally, for $m\geq 1$, we use the fact that $r\geq\tfrac14 l^{-(m+1)}$ and $d_w(l)\geq 2$, to deduce this is bounded above by $r^{d_w(l)}$. For $m=0$, we recall that $\tau^l=\tfrac12 l^{2-d_w(l)}$ (see Theorem \ref{t:mr2}) to obtain the bound $C\tau^lr^2$. In particular, in either case, the bound can be written simply as $C\Psi_l(r)\Sigma$, and so the proof is complete.
\end{proof}

We continue by applying the above Poincar\'{e} inequality to give an on-diagonal heat kernel estimate, following closely the argument of \cite{HKpcf,KusZhou}. The final claim of the following lemma is the one that we will require in order to establish the main results of this section. Note the naturalness of the appearance of $l^{-2}$ in \eqref{hkub}; this is the time taken by $X^l_{\tau^l\cdot}$ to get beyond the fractal scale.

\begin{lemma}\label{hkbound} There exists a constant $c$, independent of $l$, such that
\[\sup_{x\in \mathrm{SG}(l)}p^l_{t}(x,x)\leq \left\{
                                               \begin{array}{ll}
                                                 c(t/\tau^l)^{-d/2}, & \forall t\in[l^{-d_w(l)},\tau^l]; \\
                                                 cl^{d-d_f(l)}t^{-d_f(l)/d_w(l)}, & \forall t\in(0,l^{-d_w(l)}].
                                               \end{array}
                                             \right.\]
In particular, there exists a constant $c$, independent of $l$, such that
\begin{equation}\label{hkub}
\sup_{x\in \mathrm{SG}(l)}p^l_{\tau^lt}(x,x)\leq ct^{-d/2},\qquad \forall t\in[l^{-2},1].
\end{equation}
\end{lemma}
\begin{proof}
Let $u_0\in L^2(\mathrm{SG}(l),\mu^l)$ with $u_0 \geq 0$ and $\|u_0\|_1 = 1$. (In this proof, we suppress the dependence on $l$ in the notation of norms.) Set $u_t(x) = (P^l_t u_0)(x)$, where $(P^l_t)_{t\geq 0}$ is the semigroup associated with $X^l$, and $g(t) = \|u_t\|_2^2$; note that $g$ is continuous and decreasing. Define
\[
V_l(r)=\left\{
           \begin{array}{ll}
             l^{-d+d_f(l)}r^{d_f(l)}, & \hbox{if $r\le l^{-1}$;} \\
              r^{d}, & \hbox{if $r\ge l^{-1}$.}
           \end{array}
\right.\]
(We highlight that, from \eqref{ndl}, it is possible to check $l^{-d+d_f(l)}\rightarrow 1$ as $l\rightarrow \infty$, and so this term is bounded above and below by constants, uniformly in $l$; we include it to ensure the function $V_l$ is continuous.) Now, suppose $m\geq 0$ and $r\in  (\tfrac14l^{-(m+1)},\tfrac14l^{-m}]$. For each $i\in \{1,\dots,N_{d,l}\}^{m}$, it is possible to cover $\psi_i(\mathrm{SG}(l))$ by balls of the form $B(x,r)$ where $x\in \psi_i(\Delta)$ and, for each $j\in \{1,\dots,N_{d,l}\}^{m+1}$ with $(j_1,\dots,j_m)=(i_1,\dots,i_m)$, $\psi_j(\mathrm{SG}(l))$ has a non-empty intersection with at most a constant number $C_0$ of the sets $B(x,5r)$; the constant $C_0$ can be chosen independently of $l$, $m$ and $r$. Write $\Gamma_i$ for the centres of the balls in this cover of $\psi_i(\mathrm{SG}(l))$. Applying the self-similarity of the Dirichlet form $\mathcal{E}^l$, as follows from \eqref{edef}, we then deduce that
\begin{eqnarray}
\frac{d}{dt} g(t) &=& -2 {\mathcal E}^l(u_t,u_t)\nonumber\\
&=&-2\rho_l^{m+1}\sum_{j\in  \{1,\dots,N_{d,l}\}^{m+1}}{\mathcal E}^l(u_t\circ\psi_j,u_t\circ\psi_j)\nonumber\\
&\leq &-2C_0^{-1}\sum_{i\in  \{1,\dots,N_{d,l}\}^{m}}\sum_{x\in \Gamma_i}\rho_l^{m+1}\sum_{\substack{j\in  \{1,\dots,N_{d,l}\}^{m+1}:\\(j_1,\dots,j_m)=(i_1,\dots,i_m),\\\psi_j(\mathrm{SG}(l))\cap B(x,5r)\neq\emptyset}}{\mathcal E}^l(u_t\circ\psi_j,u_t\circ\psi_j)\nonumber \\
&\leq &-C\Psi_l(r)^{-1}\sum_{i\in  \{1,\dots,N_{d,l}\}^{m}}\sum_{x\in \Gamma_i}\int_{B(x,r)\cap\psi_i(\mathrm{SG}(l))}\left(u_t(y)-\bar{u}_t^{x,i}\right)^2\mu^l(dy),\nonumber
\end{eqnarray}
where $\bar{u}_t^{x,i}$ is the average of $u_t$ on the set $B(x,r)\cap\psi_i(\mathrm{SG}(l))$ (defined similarly to the definition of $\bar{f}^{l,i}_{B(x,r)}$ in the statement of Lemma \ref{thm:lem4.5}) and, to deduce the final equality, we have applied Lemma \ref{thm:lem4.5}. Now, we observe that the inner integral can be written
\[\int_{B(x,r)\cap\psi_i(\mathrm{SG}(l))}u_t(y)^2\mu^l(dy)-\mu^l\left(B(x,r)\cap\psi_i(\mathrm{SG}(l))\right)\left(\bar{u}_t^{x,i}\right)^2.\]
Since the sets $B(x,r)$, $x\in\Gamma_i$, $i\in   \{1,\dots,N_{d,l}\}^{m}$ form a cover of $\mathrm{SG}(l)$, it holds that
\[\sum_{i\in  \{1,\dots,N_{d,l}\}^{m}}\sum_{x\in \Gamma_i}\int_{B(x,r)\cap\psi_i(\mathrm{SG}(l))}u_t(y)^2\mu^l(dy)\geq \int_{\mathrm{SG}(l)}u_t(y)^2\mu^l(dy)=g(t).\]
Moreover, as the semigroup is conservative, $\|u_t\|_1 = 1$, which readily implies
\[\bar{u}_t^{x,i}\leq \frac{1}{\mu^l\left(B(x,r)\cap\psi_i(\mathrm{SG}(l))\right)}\leq \frac{C}{V_l(r)}.\]
Putting these estimates together yields
\[\frac{d}{dt} g(t)\leq -c_1\Psi_l(r)^{-1}\left(g(t)-c_2V_l(r)^{-1}\right),\]
for any $r\leq \tfrac14$. Therefore, if $g(t)> c_2V_l(r)^{-1}$,
\begin{equation}
 -\frac{d}{dt} \log{\left(g(t)-c_2V_l(r)^{-1}\right)} \geq 2c_1 \Psi_l(r)^{-1}. \label{eq:dineq}
\end{equation}
Next, fix  $\alpha\in (0,1)$, set $r_n:=\alpha^n$ and define $s_n = \inf \{ t\geq 0 : g(t) \leq c_2V_l(r_n)^{-1} \}$ for $n\in \mathbb{N}$. Then (\ref{eq:dineq}) holds with $r=r_n$ for $0 < t < s_n$. With $r=r_n$, integrating (\ref{eq:dineq}) over $t$ from $s_{n+2}$ to $s_{n+1}$, we thus obtain
\begin{eqnarray*}
 2c_1 \Psi_l(r_n)^{-1} \left(s_{n+1}-s_{n+2}\right) &\leq& -\log{\left(g(s_{n+1})-c_2V_l(r_n)^{-1}\right)} +
    \log{\left(g(s_{n+2})-c_2V_l(r_n)^{-1}\right)}   \\
  &=& \log{\left(\frac{c_2V_l(r_{n+2})^{-1}- c_2V_l(r_n)^{-1}}{c_2V_l(r_{n+1})^{-1}-
 c_2V_l(r_n)^{-1}}\right)}\\
 & =&\log\left({\frac{(V_l(r_n)/V_l(r_{n+2}))- 1}{(V_l(r_n)/V_l(r_{n+1}))-1}}\right)\\
 &\leq &c_3,
\end{eqnarray*}
for some $c_3>0$ independent of $l$, where the last inequality is due to the fact that
\[\alpha^{-1}\le \alpha^{-d_f(l)}\le V_l(r_n)/V_l(r_{n+1})\le \alpha^{-d},\]
which can easily be obtained from the definition of $V_l(r)$ and the observation that $d_f(l)\ge 1$. Hence we have shown that $s_{n+1} - s_{n+2} \leq c_4  \Psi_l(r_n)$, and iterating this gives
\[s_n \leq c_4 \sum_{k=n-1}^{\infty}  \Psi_l(r_k) \leq c_5  \Psi_l(r_n).\]
This implies that $g(c_5  \Psi_l(r_n)) \leq g(s_n) =c_2V_l(r_n)^{-1}$. It follows in turn that if $c_5  \Psi_l(r_{n+1}) \leq t < c_5  \Psi_l(r_n)$, then
\[ g(t) \leq  c_2V_l(r_n)^{-1} \leq \left\{\begin{array}{ll}
c_6(t/\tau^l)^{-d/2}, & \forall t\in[l^{-d_w(l)},\tau^l]; \\
 c_6l^{d-d_f(l)}t^{-d_f(l)/d_w(l)}, & \forall t\in(0,l^{-d_w(l)}].
  \end{array}\right.\]
(Here, we note that by Remark \ref{HKcomp}(a), it holds that $c_7\Psi_l(r_{n})\le  \Psi_l(r_{n+1})$ for some $c_7>0$ which is independent of $l$.) Using the fact that $\|P_t^l\|_{1\rightarrow\infty} = \|P_t^l\|_{1\rightarrow2}^2$, we deduce the desired heat kernel upper bound.
\end{proof}

Now the basic preparations are in place, we are ready to start developing continuity results. We proceed by applying the elliptic Harnack inequality of Lemma \ref{L.R2} to give a uniform H\"{o}lder continuity property for harmonic functions.

\begin{theorem}\label{T.R1}
There exist constants $c_1$ and $\gamma'>0$, independent of $l$, such that if $h_l$ is bounded and harmonic for $({\cal E}_{l}, {\cal F}_{l})$ in a ball $B(x_0,2r)$ for some $x_0\in \mathrm{SG}(l)$ and $r\in (8l^{-1},\tfrac14)$, then
\begin{equation}\label{C03}
\left|h_l(x)-h_l(y)\right|\leq c_1\left(\frac{|x-y|\vee l^{-1}}{r}\right)^{\gamma'}\|h_l\|_\infty, \qquad \forall x,y\in B(x_0,r)\cap \mathrm{SG}(l).
\end{equation}
\end{theorem}
\begin{proof}
We follow the arguments in \cite[Lemma 3.8, Theorem 3.9]{BB1}. Let $h:=h_l$ be a bounded harmonic function for $({\cal E}_{l}, {\cal F}_{l})$ in $B(x_0,2r)$ for some $x_0\in \mathrm{SG}(l)$ and $r\in (8l^{-1},\tfrac14)$. For $m\in {\mathbb N}\cup\{0\}$ and $x\in B(x_0,r)\cap \mathrm{SG}(l)$, set
\[O_m (x,h):=\sup_{y\in B(x,2^{-m}r)\cap \mathrm{SG}(l)}h(y)-\inf_{y\in B(x,2^{-m}r)\cap \mathrm{SG}(l)}h(y).\]
We claim that there exists a constant $p\in (0,1)$, independent of $l$, $x$, $x_0$, $r$, $m$ and $h$, such that
\begin{equation}\label{osc-ineqa}
O_{m+1} (x,h)\le p O_m (x,h),
\end{equation}
at least whenever $2^{-(m+1)}r\geq 8l^{-1}$. When proving \eqref{osc-ineqa}, by rescaling, we may assume $-1\le h \le 1$ and $O_m (x,h)=2$. Define
\begin{eqnarray*}
 T&:=&\inf\left\{t\ge 0: X^{l}_t\in \partial B(x,2^{-m}r)\cap \mathrm{SG}(l)\right\},\\
 A&:=&\left\{y\in \partial B(x,2^{-m}r)\cap \mathrm{SG}(l): h(y)\le 0\right\}.
\end{eqnarray*}
Clearly either ${\mathbf P}^l_x( X^{l}_T\in A)\ge 1/2$ or ${\mathbf P}^l_x( X^{l}_T\in A^c)\ge 1/2$ holds; without loss of generality we may assume (by multiplying $h$ by $-1$ if needed) the former option holds. Then, for $y\in B(x,2^{-m-1}r)\cap \mathrm{SG}(l)$,
\begin{eqnarray*}
 h(y)&=&{\mathbf E}^l_y\left(h(X^{l}_T)\right)\\
 &=&{\mathbf E}^l_y\left(h(X^{l}_T): X^{l}_T\in A^c\right)+{\mathbf E}^l_y\left(h(X^{l}_T): X^{l}_T\in A\right)\\
 &\le & {\mathbf P}^l_y\left( X^{l}_T\in A^c\right)\\
 &=&1-{\mathbf P}^l_y\left( X^{l}_T\in A\right).
\end{eqnarray*}
Now, since ${\mathbf P}^l_\cdot( X^{l}_T\in A)$ is harmonic on $B(x,2^{-m}r)\cap \mathrm{SG}(l)$, when $2^{-(m+1)}r\geq 8l^{-1}$, we can apply Lemma \ref{L.R2}(ii) to obtain that
\[{\mathbf P}^l_y\left( X^{l}_T\in A\right)\ge \theta^{-1} {\mathbf P}^l_x\left( X^{l}_T\in A\right)\ge (2\theta)^{-1}\]
for some $\theta\geq 1$. This yields that $h(y)\le 1- (2\theta)^{-1}$, and hence
\[ O_{m+1} (x,h)\le 1- \frac 1{2\theta}-(-1)=2- \frac 1{2\theta}=pO_{m} (x,h),\]
where $p:=1- (4\theta)^{-1}\in (0,1)$. This proves \eqref{osc-ineqa}.

We now prove \eqref{C03}. Let $x,y\in B(x_0,r)\cap \mathrm{SG}(l)$. Suppose that there exists an $m\in {\mathbb N}\cup\{0\}$ so that
$8l^{-1}\leq 2^{-(m+1)}r\le |x-y|< 2^{-m}r$ (which is possible if $y\in B(x,r)\backslash B(x,16l^{-1})$). Then, using \eqref{osc-ineqa},
\[|h(x)-h(y)|\le  O_{m} (x,h)\le p^m  O_{0} (x,h)\le 2p^m\|h\|_\infty\le c_1\left(\frac {|x-y|}r\right)^\gamma\|h\|_\infty,\]
where $\gamma'=\log p^{-1}/\log 2$ and $c_1=2^{1+\gamma}$. Moreover, if $y\in B(x,16l^{-1})\cap \mathrm{SG}(l)$, then one can simply note that $h_l$ takes its maximum and minimum values in $B(x,16l^{-1})\cap \mathrm{SG}(l)$ on the boundary of this set, and thus the difference between $h(x)$ and $h(y)$ is also obtained there. Together with the above bound, this observation implies
\[|h(x)-h(y)|\le  c_1\left(\frac {|x-y|\vee l^{-1}}r\right)^\gamma\|h\|_\infty.\]
For $y\not\in B(x,r)$, we apply the triangle inequality to deduce that $|h(x)-h(y)|\leq |h(x)-h(x_0)|+|h(y)-h(x_0)|$, and then apply the above estimate to each part.
\end{proof}

Following the arguments in \cite[Section 3]{BKK}, our subsequent goal is to show the $\lambda$-potentials associated with $X^l$, $l\geq 2$, are uniformly H\"older continuous. Specifically, we define
\[U^l_\lambda f(x)={\mathbf E}^l_x\int_0^\infty e^{-\lambda t} f(X^{l}_t)\, dt.\]

\begin{propn}\label{PR3}
There exist constants $c_1$ and $\gamma$, independent of $l$ and $\lambda$, such that if $f_l:\mathrm{SG}(l)\rightarrow\mathbb{R}$ is a bounded, measurable function, then
\[\left|U^l_\lambda f_l(x)-U^l_\lambda f_l(y)\right|\leq c_1 \left(\tau^l+\lambda^{-1}\right) \left(\left|x-y\right|\vee l^{-1}\right)^{\gamma} \|f_l\|_\infty,\qquad\forall x,y\in\mathrm{SG}(l).\]
\end{propn}
\begin{proof}
Fix $x_0\in \mathrm{SG}(l)$  and $r\in (8l^{-1},\tfrac18)$. Suppose $x,y\in B(x_0,r)\cap \mathrm{SG}(l)$. Define $\bar{\sigma}_r:= \bar{\sigma}^l(x_0,2r)$ and $f:=f_l$. By applying the strong Markov property at time $\bar{\sigma}_r$,
\begin{align*}
U^l_\lambda f(x)&={\mathbf E}^l_x\int_0^{\bar{\sigma}_r} e^{-\lambda t} f(X^{l}_t) \, dt
+{\mathbf E}^l_x \left((e^{-\lambda \bar{\sigma}_r}-1)U^l_\lambda f(X^{l}_{\bar{\sigma}_r})\right) + {\mathbf E}^l_x U^l_\lambda f(X^{l}_{\bar{\sigma}_r})\\
&=I_1+I_2+I_3,
\end{align*}
and similarly when $x$ is replaced by $y$. By Lemma \ref{L.R2}, we have
\[|I_1|\leq \|f\|_\infty {\mathbf E}^l_x \bar{\sigma}_r\leq c\Psi_l(r)\|f\|_\infty.\]
Moreover, by the observation that $|1-e^{-u}|\leq u$ for all $u\geq 0$ and Lemma \ref{L.R2},
\[|I_2|\leq \lambda {\mathbf E}^l_x \bar{\sigma}_r \|U^l_\lambda f\|_\infty\leq c\Psi_l(r) \|f\|_\infty.\]
Similar bounds hold when $x$ is replaced by $y$, and so we deduce that
\begin{equation}\label{C05}
\left|U^l_\lambda f(x)-U^l_\lambda f(y)\right|\leq c\Psi_l(r) \| f\|_\infty+\left|{\mathbf E}^l_x U^l_\lambda f(X^{l}_{\bar{\sigma}_r})-{\mathbf E}^l_y U^l_\lambda f(X^{l}_{\bar{\sigma}_r})\right|.
\end{equation}
Now, the function $z\to {\mathbf E}^l_z U^l_\lambda f(X^{l}_{\bar{\sigma}_r})$ is bounded and harmonic for $({\cal E}_{l}, {\cal F}_{l})$
in $B(x_0,2r)\cap \mathrm{SG}(l)$, and thus, by Theorem \ref{T.R1}, the second term in \eqref{C05} is bounded by
$$c\left(\frac{|x-y|\vee l^{-1}}{r}\right)^{\gamma'} \|U^l_\lambda f\|_\infty\leq c\lambda^{-1}\left(\frac{|x-y|\vee l^{-1}}{r}\right)^{\gamma'} \| f\|_\infty.$$
Summarising, we have shown that
\[\left|U^l_\lambda f(x)-U^l_\lambda f(y)\right|\leq c\left(\Psi_l(r)+\lambda^{-1}\left(\frac{|x-y|\vee l^{-1}}{r}\right)^{\gamma'}\right)\]
for any $x,y\in B(x_0,r)\cap \mathrm{SG}(l)$, where $x_0\in \mathrm{SG}(l)$ and $r\in (8l^{-1},\tfrac18)$. Moreover, by applying the triangle inequality, it is straightforward to extend the range of $r$ from $(8l^{-1},\tfrac18)$ to $(8l^{-1},r_0]$ for any $r_0>0$; in particular, we take $r_0=8$. Hence, for $x,y\in\mathrm{SG}(l)$, we can take $x_0=x$ and $r=(|x-y|\vee 8^{(2+\gamma')/\gamma'}l^{-1})^{\gamma'/(2+\gamma')}\in(8l^{-1},r_0)$ to conclude
\begin{eqnarray*}
\lefteqn{\left|U^l_\lambda f(x)-U^l_\lambda f(y)\right|}\\&\leq &c\left(\Psi_l\left(\left(|x-y|\vee 8^{(2+\gamma')/\gamma'}l^{-1}\right)^{\gamma'/(2+\gamma')}\right)+\lambda^{-1}\left(|x-y|\vee l^{-1}\right)^{2\gamma'/(2+\gamma')}\right)\|f\|_\infty\\
&\leq & c\left(\tau^l+\lambda^{-1}\right) \left(|x-y|\vee l^{-1}\right)^{2'\gamma/(2+\gamma')}\|f\|_\infty,
\end{eqnarray*}
which completes the proof.
\end{proof}

We transfer Proposition \ref{PR3} to the semigroup by appealing to spectral theory. In particular, for each $l\geq 2$, let $(P_t^{l})_{t\geq 0}$ be the semigroup of $X^l$. By the standard spectral theorem, there exist projection operators $(E_\nu^{l})_{\nu\geq 0}$ on the space $L^2(\mathrm{SG}(l), \mu^l)$ such that: for $f\in L^2(\mathrm{SG}(l), \mu^l)$,
\begin{align}
f&=\int_0^\infty \, dE^{l}_\nu(f),\nonumber\\
P^{l}_tf&=\int_0^\infty e^{-\nu t} \, dE^{l}_\nu(f),\nonumber\\
 U^l_\lambda f&=\int_0^\infty \frac{1}{\lambda+\nu} \, dE^{l}_\nu(f).\label{C07}
\end{align}
(Note that the integration above is with respect to $\nu$.)

\begin{propn}\label{PR4}
If $f_l$ is in $L^2(\mathrm{SG}(l), \mu^l)$, then $P^{l}_{\tau^l t}f_l$ is almost-surely equal to a function that is H\"older continuous. Specifically, for this version of $P^{l}_{\tau^l t}f_l$, it holds that
\[\left|P^{l}_{\tau^l t}f_l (x)-P^{l}_{\tau^l t}f_l (y)\right|\leq Ct^{-(1+d/2)}\left(|x-y|\vee l^{-1}\right)^{\gamma} \|f_l\|_2,\qquad\forall x,y\in\mathrm{SG}(l),\:t\in[2l^{-2},1],\]
where $\gamma$ is the constant of Proposition \ref{PR3} and $C$ is a constant independent of all the other variables.
\end{propn}
\begin{proof}
Write $\langle f,g \rangle$ for the inner product in $L^2(\mathrm{SG}(l), \mu^l)$. Note that in
what follows $t$ is fixed. As before, we write $f=f_l$. Define
\[h=\int_0^\infty \left((\tau^l)^{-1}+\nu\right)e^{-\nu \tau^l t}\, dE^{l}_\nu(f).\]
Since $\sup_\nu ((\tau^l)^{-1}+\nu)^2 e^{-2\nu \tau^l t}\leq 1+(\tau^l t)^{-2}=:c_{1,t}(l)$, then
\[\int_0^\infty \left((\tau^l)^{-1}+\nu\right)^2 e^{-2\nu\tau^l t} \, d\langle E^{l}_\nu(f), E^{l}_\nu(f)\rangle
\leq c_{1,t}(l)\int_0^\infty \, d\langle E^{(l)}_\nu(f), E^{(l)}_\nu(f)\rangle=c_{1,t}(l)\|f\|_2^2,\]
and so we see that $h$ is an element of $L^2(\mathrm{SG}(l), \mu^l)$.

Next, suppose $g\in L^1(\mathrm{SG}(l), \mu^l)$. Then $\|P^{l}_{\tau^l t} g\|_1\leq \|g\|_1$, and
\[\left|P^{l}_{\tau^l t}g(x)\right|=\left|\int p_{\tau^l t}^l(x,y)g(y)\, \mu_l(dy)\right|\leq c_{2,t}(l)\|g\|_1,\]
where
\[c_{2,t}(l):= \sup_{x,y\in\mathrm{SG}(l)} p_{\tau^l t}^{l}(x,y).\]
Note that this constant is finite by the continuity of the heat kernel. It follows that
\[\|P^{l}_{\tau^l t}g\|_2\leq \|P^{l}_{\tau^l t}g\|_\infty\leq c_{2,t}(l)\|g\|_1.\]
Using Cauchy-Schwarz and the fact that $\sup_\nu ((\tau^l)^{-1}+\nu) e^{-\nu \tau^l t}=c_{1,t}(l)^{1/2}<\infty,$ we have
\begin{align*}
\langle h,g \rangle&=\int_0^\infty \left((\tau^l)^{-1}+\nu\right)e^{-\nu \tau^l t}\, d\langle E^{l}_\nu(f), E^{l}_\nu(g)\rangle\\
&\leq \left(\int_0^\infty \left((\tau^l)^{-1}+\nu\right)^2e^{-\nu \tau^l t}\, d\langle E^{l}_\nu(f), E^{l}_\nu(f)\rangle\right)^{1/2}\times\left(\int_0^\infty e^{-\nu \tau^l t} \, d\langle E^{l}_\nu(g), E^{l}_\nu(g)\rangle\right)^{1/2}\\
&\leq  \left(c_{1,t/2}(l)\int_0^\infty  \, d\langle E^{l}_\nu(f), E^{l}_\nu(f)\rangle\right)^{1/2}\times\left(\int_0^\infty  e^{-\nu\tau^l t}  \, d\langle E^{l}_\nu(g), E^{l}_\nu(g)\rangle\right)^{1/2}\\
&=c_{1,t/2}(l)^{1/2} \|f\|_2\|P^{l}_{\tau^l t/2}g\|_2\\
&\leq {c}'_t(l) \|f\|_2 \|g\|_1,
\end{align*}
where ${c}'_t(l):=c_{1,t/2}^{1/2}c_{2,t/2}(l)$. By considering the supremum over $g\in L^1(\mathrm{SG}(l), \mu^l)$ with $L^1$-norm less than $1$, we obtain that $\|h\|_\infty\leq c_{t}'(l)\|f\|_2$.
Additionally, by \eqref{C07},
 \[U^l_{(\tau^l)^{-1}} h=
\int_0^\infty e^{-\nu \tau^l t} \, dE^{l}_\nu(f)=P^{l}_{\tau^lt}f,\]
and so the H\"older continuity of $P^{l}_{\tau^lt}f$ follows from Proposition \ref{PR3}. In particular, we have that the constant of the result arises from the following bound: for $t\in [2l^{-2},1]$,
\[2c_1\tau^l{c}'_t(l)\leq 2c_1\left(1+t^{-1}\right)\sup_{x,y\in\mathrm{SG}(l)} p_{\tau^l t/2}^{l}(x,y)\leq Ct^{-(1+d/2)},\]
where $c_1$ is constant of Proposition \ref{PR3} and we have applied Lemma \ref{hkbound} to deduce the final inequality.
\end{proof}

We are now in a position to prove Theorems \ref{PT6} and \ref{PT7}.

\begin{proof}[Proof of Theorems \ref{PT6} and \ref{PT7}]
Fix $y\in\mathrm{SG}(l)$ and let $f_l(z)=p_{\tau^lt/2}^l(z,y)$. Since
\[P^{l}_{\tau^lt/2}f_l(x)=\int p^l_{\tau^lt/2}(x,z) f_l(z)\, \mu^l(dz)=\int p^l_{\tau^lt/2}(x,z)p^l_{\tau^lt/2}(z,y)\, \mu^l(dz)
=p^l_{\tau^lt}(x,y),\]
we obtain from Proposition \ref{PR4} that
\begin{align*}
\left|p_{\tau^lt}^l(x,y)-p_{\tau^lt}^l(x',y)\right|&=\left|P^{l}_{\tau^lt/2}f_l(x)-P^{l}_{\tau^lt/2}f_l(x')\right|\\
&\leq Ct^{-(1+d/2)}\left(|x-x'|\vee l^{-1}\right)^{\gamma} \|f_l\|_2,\\
&=Ct^{-(1+d/2)}\left(|x-x'|\vee l^{-1}\right)^{\gamma} p_{\tau^lt}^l(y,y)^{1/2},
\end{align*}
for all $x,y\in\mathrm{SG}(l)$, $t\in[4l^{-2},1]$. By Lemma \ref{hkbound}, the term $p_{\tau^lt}^l(y,y)^{1/2}$ is bounded above by $Ct^{-d/4}$, and thus the result of Theorem \ref{PT7} follows for the range $t\in [4l^{-2},1]$. By simple modifications of the previous arguments, the constant 1 is readily replaced by an arbitrary constant $t_0$.

Finally, let $I=[t_1,t_0]$, where $0<t_1<t_0$. Assume that $l$ is large enough so that $4l^{-2}\leq t_1$. Then, from Theorem \ref{PT7},
\[\sup_{\substack{x,x',y,y' \in \mathrm{SG}(l):\\|x-x'|,|y-y'|\leq \delta}}\sup_{t\in I}\left|p^l_{\tau^l t}(x,y)-p^l_{\tau^l t}(x',y')\right|\leq
C\sup_{t\in I}t^{-(1+3d/4)}\left(\delta\vee l^{-1}\right)^{\gamma}.\]
Taking the limit as $l\rightarrow\infty$ and then $\delta\rightarrow0$ yields Theorem \ref{PT6}, as desired.
\end{proof}

\section{Other examples}\label{sec5}

\subsection{Triangles}

We have worked with the family $\mathrm{SG}(l)$, which has a part of the triangular lattice as the initial structure in each fractal. This could be modified to, for example, a sequence of triangles attached to the outer edge of the unit triangle, see Figure \ref{triangle}. This would have a mass of $3(n-1)$ triangles of side length $1/n$ at stage $n$. In this case, the fractal dimension of the limit would be 1 and we should see that the associated diffusions converge to a one-dimensional Brownian motion on the boundary of a unit triangle. (Actually, for this example, the limiting process is described by a resistance form, and so we could alternatively deduce this conclusion by applying the convergence result of \cite[Theorem 1.3]{CHK}. Moreover, the corresponding local limit theorem would follow by checking the equicontinuity of the heat kernel using the resistance-based approach of \cite[Section 4]{CrHPA}.) Other higher-dimensional analogues of the construction would also be possible. For instance, in three-dimensions, taking those first-stage tetrahedra touching the outer boundary of the initial one would yield a two-dimensional Brownian motion on the boundary of the tetrahedron. Or, taking only those first-stage tetrahedra intersecting the edges of the initial one would yield a one-dimensional Brownian motion on the edges of the tetrahedron. We believe that straightforward modifications to the arguments of this paper would yield the corresponding results in these cases.

\begin{figure}[t]
  \centering
  \includegraphics[width=0.5\textwidth]{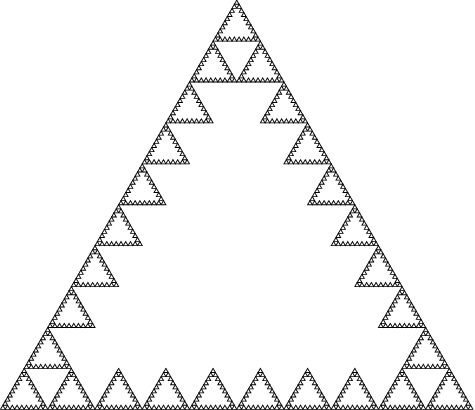}
  \caption{A variation on the $\mathrm{SG}(l)$ containing only the triangles touching the outer boundary of $\Delta$, shown at the second stage of construction for $l=10$.}\label{triangle}
\end{figure}

\subsection{Vicsek sets}

We discuss here how our analysis can also be transferred to a family of finitely ramified fractals based on the square and the cube. These we call Vicsek sets and label them VS$(l)$, where $l=2n+1$ is the number of divisions of the side, see Figure \ref{vsfig}. There is much that is the same as in the case of the Sierpinski gasket, but we need to be a little more careful in certain places.

\begin{figure}[t]
  \centering
  \includegraphics[width=0.3\textwidth]{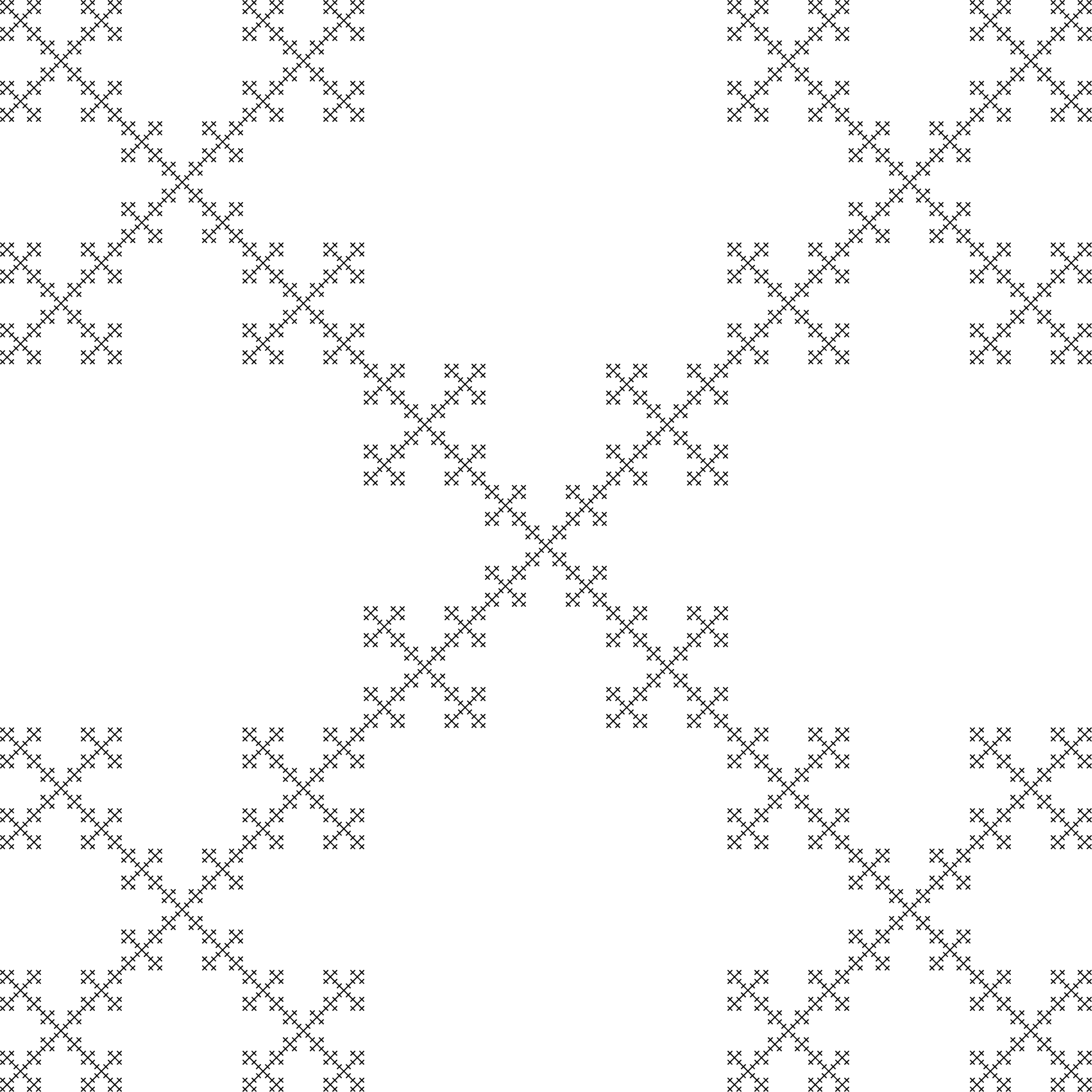}\hspace{5pt}\includegraphics[width=0.3\textwidth]{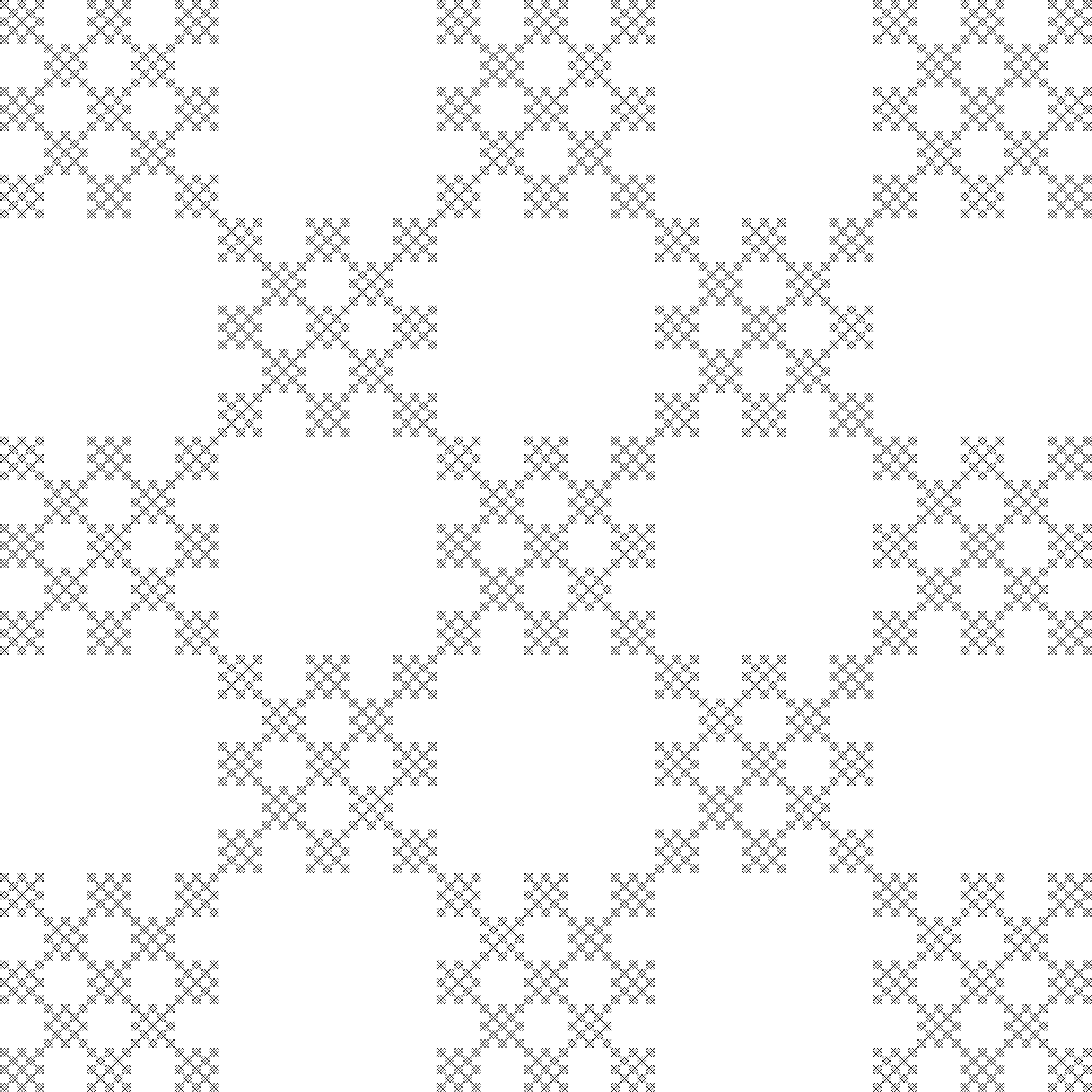}\hspace{5pt}\includegraphics[width=0.3\textwidth]{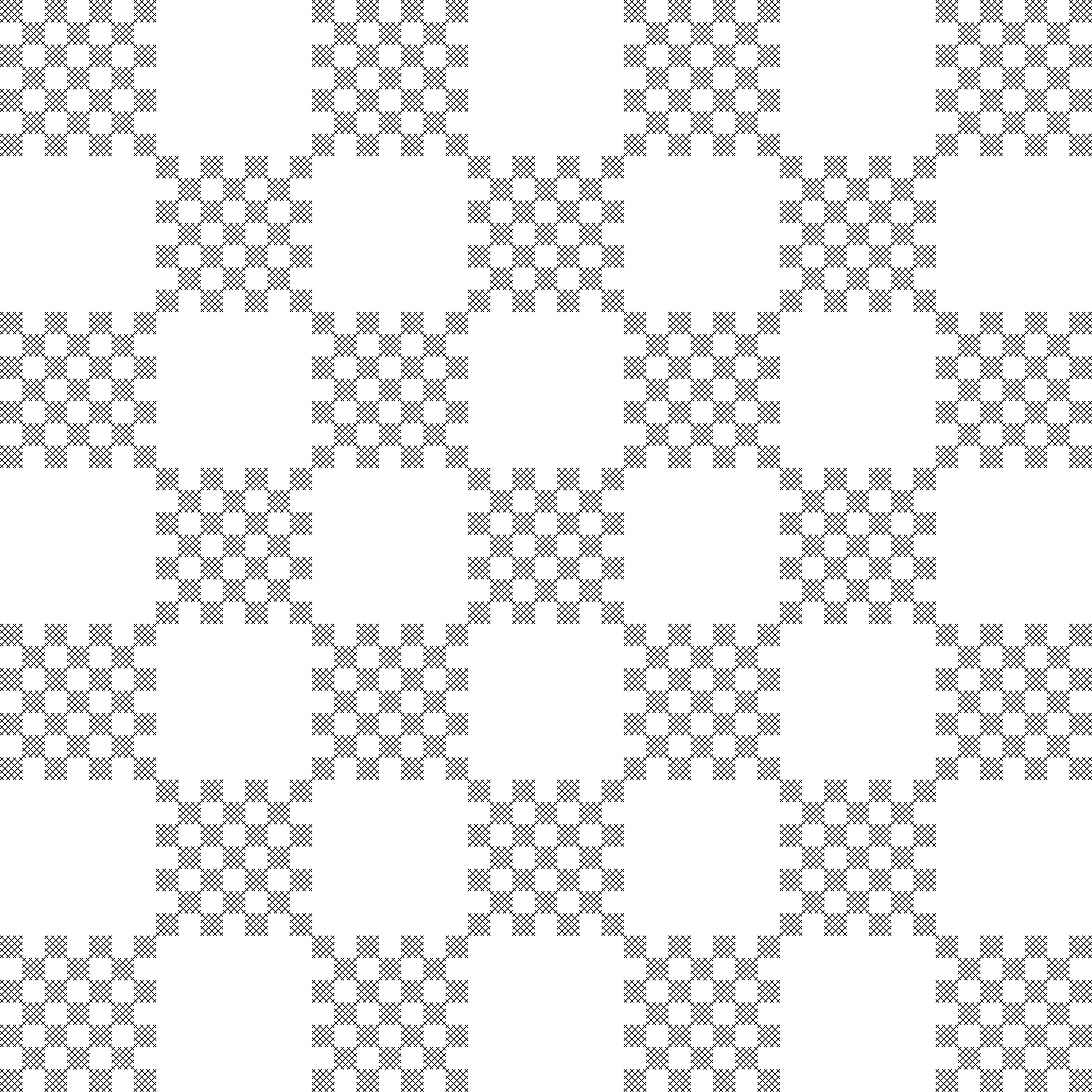}
  \caption{The sets $\mathrm{VS}(2)$, $\mathrm{VS}(3)$ and $\mathrm{VS}(4)$ in two dimensions, shown at the 5th, 4th and 3rd stages of construction, respectively.}\label{vsfig}
\end{figure}

We start by recalling the resistance estimates from \cite{HK}.

\begin{lemma}[{\cite[Section~4.1]{HK}}] If $d=2$, then $\rho_l \asymp \log l$. If $d\geq 3$, then $\rho_l \asymp 1$.
\end{lemma}

\subsubsection{Two-dimensional case}

The construction is to take a unit square and divide it into a $l \times l$ grid. We then construct a checkerboard of squares of side $1/l$. Now remove all the squares which are connected diagonally, but not to one of the outer corners. Repeat this process within each remaining square to build a fractal with fractal dimension $\log((l^2+1)/2)/\log(l)$. We see that as $l\to\infty$, the fractal dimension converges to 2 as expected.

The first step is to show the existence of diffusions on the fractals in the sequence. For the case of the square, this is a consequence of the fact that it is a family of nested fractals. By the argument of Lindstrom \cite{Lin}, for each $n$, there exists a fixed point $w_l$ such that the network with conductance $1$ on the horizontal and vertical edges and $w_l$ on the diagonal edges, will satisfy the fixed point problem for the Dirichlet form
\[ \mathcal{E}^{(0)}_l(f,f) = \sum_{i=1}^{(l^2+1)/2} \rho_l \mathcal{E}^{(0)}_{l,i} (f \circ \phi_i, f \circ \phi_i). \]
(The graph upon which the Dirichlet form $\mathcal{E}^{(0)}_l$ is based is shown in Figure \ref{vgfig}.) From this fixed point, we can construct a sequence of Dirichlet forms $(\mathcal{E}_l, \mathcal{F}_l)$ each defined on VS$(l)$. For each $l$, the corresponding Hunt process $(X_t^l)_{t\geq 0}$ is the Brownian motion on the VS$(l)$. We now let $l\to\infty$. It is shown in \cite{HK} that
\[ \rho_l \asymp \log{l}. \]
This makes it straightforward to follow the same steps as for the Sierpinski gasket case to show that there will be a sequence of scaling parameters $\tau^l$ such that as $l\to\infty$
\[ (X^l_{\tau^lt})_{t\geq 0} \to (B_t)_{t\geq 0},\]
where $B$ is reflected Brownian motion on the unit square.

\begin{figure}[t]
  \centering
  \includegraphics[width=0.2\textwidth]{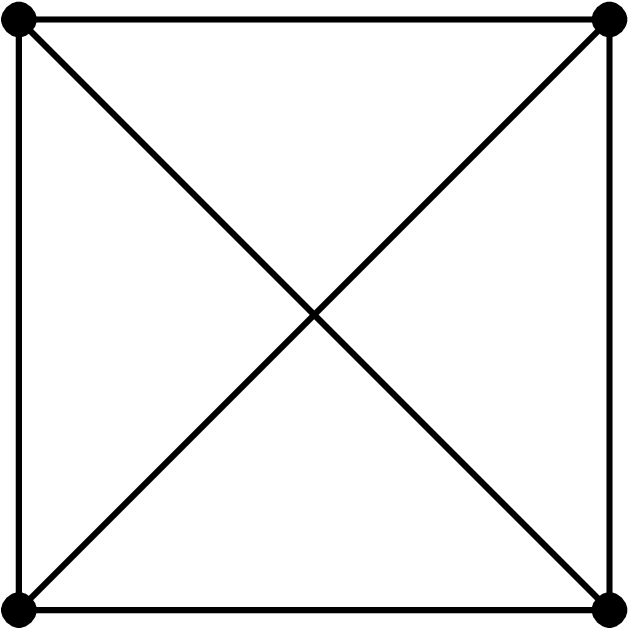}
  \caption{The graph upon which $\mathcal{E}^{(0)}_l$ is built for the two-dimensional Vicsek set.}\label{vgfig}
\end{figure}

\subsubsection{Three-dimensional case}

We consider the case of cubes where we break an initial cube up into a collection of $(2n+1)^3$ cubes of side length $1/(2n+1)$. The cubes that are connected diagonally, but not to an outer corner, are removed. The dimension of the resulting fractal is easily seen to be $\log((n+1)^3+n^3)/\log(2n+1)$, which converges to $3$ as expected.

The first issue is the existence of a Brownian motion on these sets as they are not nested fractals in dimension 3 (or indeed higher dimensions), see the lecture notes \cite{Barlow} for a discussion of the class of nested fractals. However, the structure of these sets is such that we can apply the same arguments as used by \cite{Lin} to show existence.

Let $\mathcal{P} = \{(p_1,p_2,p_3): p_1\geq p_2\geq p_3\geq 0, 3p_1+3p_2+p_3=1\}$
be the set of probabilities for a random walk on the unit cube where $p_1$ is the probability of crossing an edge, $p_2$ the probability of crossing a diagonal on a face and $p_3$ is the probability of crossing the main diagonal across the cube. Note that these probabilities are proportional to the conductivities $(c_1,c_2,c_3)$ between the vertices. That is, if we fix the edge conductivity $c_1=1$, then for the diagonal across a face $c_2=p_2/p_1$ and, for the main diagonal across the cube, $c_3=p_3/p_1$.

\begin{propn}\label{renorm}
For each $l$, there is a fixed point for the renormalization map in the set $\mathcal{P}$. (In particular, one can construct a sequence of consistent resistance metrics on the different levels of the fractal, analogous to \eqref{consist}.)
\end{propn}

\begin{proof}
This follows the same argument as \cite{Lin}. We show that the map on transition probabilities induced by considering the trace of the random walk on $V_1^l$ onto the vertex set $V_0$ is a continuous map from $\mathcal{P}$ to $\mathcal{P}$ and hence, by Brouwer's fixed point theorem,
the fixed point exists. (The map on the probabilities is clearly continuous in $p$.)

Let $X^{l,1}$ denote the Markov chain on $V^l_1$ moving according to the transition probabilities given by an element of $\mathcal{P}$. The renormalization map is equivalent to the map on transition probabilities given by $f_i(p) = \mathbb{P}^0(X^{l,1}_{T_1}=A_i)$, where $T_1 = \inf\{n\geq 0: X^{l,1}\in V_0\backslash\{0\}\}$ and the vertices $A_i$ are as in Figure~\ref{fig:cube}. Now, by considering the paths the random walk can take between the vertices and using the reflection symmetry, it is possible to check that, if the vertex is further away from 0, the probability of reaching it first will be smaller. To see this for the edge and the diagonal across a face, we consider the bisector of the face. To move from 0 to $A_2$, we must cross the bisector, shown as the dotted line in the figure, so taking a reflection in this plane we can map any path from 0 to $A_2$ to a path from 0 to $A_1$. Hence, not only are there are more paths in $V^l_1$ to $A_1$, those that hit the bisector have a higher probability than the corresponding reflected path that goes from $0$ to $A_2$. Thus the probability $f_1(p)$ of arriving at $A_1$ must be greater then the probability $f_2(p)$ of arriving at $A_2$. Similarly, for the long diagonal, we can take the plane bisecting $A_2$ and $A_3$, as indicated on the figure by the dashed line, and any path on $V^l_1$ from 0 to $A_3$ can be reflected into a path from 0 to $A_2$, which therefore has a higher probability of being reached, i.e.\ $f_3(p)<f_2(p)$.
\end{proof}

\begin{figure}
\centering
\includegraphics[width=0.3\textwidth]{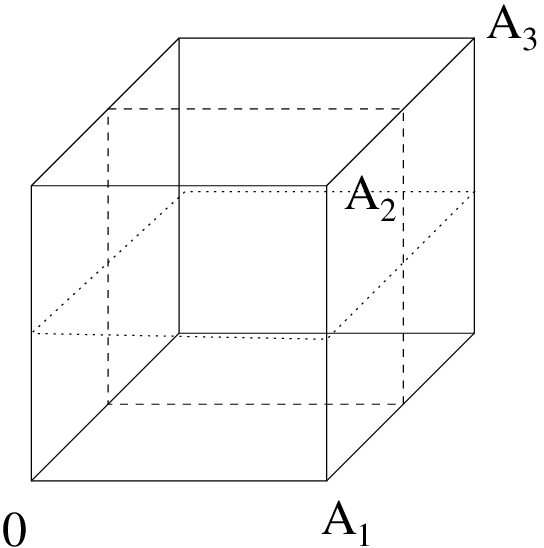}
\caption{The cube with corners labelled as in the proof of Proposition \ref{renorm}.}\label{fig:cube}
\end{figure}

Given Proposition \ref{renorm}, we have the existence of the processes exactly as for nested fractals. The next task is to prove that the resistance scale factor $\rho_l$ is strictly larger than one.

\begin{lemma} \label{vsrholem} For $d=3$,
\[\inf_{l\geq 2}\rho_l>1.\]
\end{lemma}
\begin{proof} We follow an approach similar to the Sierpinski gasket case. For each $l$, suppose $V_1^l$ is the vertex set of a graph, equipped with all edges within each 0-cell. Moreover, suppose $(R_{V_1^l}(x,y))_{x,y\in V_1^l}$ is the effective resistance on $V_1^l$ when edges are equipped with the fixed point resistors, and define $R_{V_0}$ similarly on $V_0$. To fix the normalisation, we consider the case when the fixed point resistances on the `edges' (as opposed to the `diagonals') are given by 1. As in the proof of Lemma \ref{rhobound}, we will short all the vertices at a fixed graph distance from 0.

As the resistance scale factor will be independent of which edge is considered, we focus on the main diagonal. Using the labelling from Figure~\ref{fig:cube}, we note that
\[ R_{V_1^l}(0,A_3) = \rho_l R_{V_0}(0,A_3). \]

Firstly, write $a$ for the resistance along the diagonal across a face and $b$ for the resistance along the long diagonal across the cube. (Recall that the resistance along an edge is equal to 1.) By a standard calculation we get
\[ R_{V_0}(0,A_3)= \frac{1}{2}\frac{b(5ab+2a+b)}{(a+2b+ab)(3b+1)}. \]

To prove the result, we need a good lower bound on $R_{V_1^l}(0,A_3)$. Suppose $B^l_0$ and $B^l_1$ are the vertices in $V_1^l$ at graph distance $\lfloor l/2 \rfloor$ from 0 and $A_3$, respectively. (Similar to Figure \ref{vl1fig} for the Sierpinski gasket.) By Rayleigh's monotonicity principle and the parallel law,
\[ R_{V_1^l}(0,A_3)\geq R_{V_1^l}\left(0,B^l_0\right)+R_{V_1^l}\left(B_1^l,A_3\right)=2 R_{V_1^l}\left(0,B_0^l\right).\]
By shorting vertices at each fixed graph distance from 0, we have
\begin{eqnarray*}
R_{V_1^l}\left(0,B_0^l\right)&\geq & \sum_{i=1}^{\lfloor (l-3)/4 \rfloor}\frac{2}{((i+1)^3-i^3-1)(4+8/a+4/b)+1+3/a+3/b}\\
&\geq & \frac{1}{2(1+2/a+1/b)}
\sum_{i=1}^{\lfloor (l-3)/4 \rfloor} \frac{1}{(i+1)^3-i^3} \\
&\to & \frac{1}{2(1+2/a+1/b)} \frac{\pi}{\sqrt{3}}\tanh\left(\frac{\pi}{2\sqrt{3}}\right), \;\; l\to\infty.
\end{eqnarray*}
Writing
\[ \eta(a,b) = \frac{1}{2R_{V_0}(0,A_3)(1+2/a+1/b)}, \]
we have
\[ \liminf_{l\to\infty} \rho_l \geq
\frac{2\pi}{\sqrt{3}}\tanh\left(\frac{\pi}{2\sqrt{3}}\right) \eta(a,b). \]
Note that, for any $a,b\geq 1$, it holds that $1/2 =\eta(1,1) \leq \eta(a,b) \leq \lim_{a',b'\to\infty}\eta(a',b') = 3/5$ and it is also the case that $\frac{\pi}{\sqrt{3}}\tanh\left(\frac{\pi}{2\sqrt{3}}\right)\approx 1.305.$ Combining the above observations, as  $b\geq a \geq 1$ from the fixed point argument, we deduce that
\[\liminf_{l\rightarrow\infty}\rho_l>1.\]
Since $\rho_l>1$ for each fixed $l$, this completes the proof.
\end{proof}

Applying Lemma \ref{vsrholem}, the following is established in exactly the same way as for Lemma \ref{diambound}.

\begin{lemma} For $d=2,3$,
\[\sup_{l\geq 2}\sup_{x,y\in \mathrm{VS}(l)}R^l(x,y)<\infty.\]
\end{lemma}

Once we have this result, we have the convergence of the processes in the same way as for the Sierpinski gasket.

\begin{remark}
We can approximate the Sierpinski carpet by a sequence of Vicsek sets by taking the side lengths as $3^l$ and removing the appropriate pieces of checkerboard to approximate the carpet. We would hope to be able to recover Brownian motion on the carpet as a limit of these processes. In two-dimensions the existence of fixed points for the sequence holds as the approximating sequence consists of nested fractals.
Establishing convergence would need some proof of the limit $\rho_{3^l} \to \rho_{SC}$, which may be more challenging. (Again, an alternative approach for this example, in which the limiting process is described by a resistance form, would be to apply \cite[Theorem 1.3]{CHK}.)
\end{remark}

\section*{Acknowledgements}

This research was supported by JSPS Grant-in-Aid for Scientific Research (A) 22H00099, 23KK0050, (C) 19K03540, 24K06758, and the Research Institute for Mathematical Sciences, an International Joint Usage/Research Center located in Kyoto University.

\bibliographystyle{amsplain}
\bibliography{biblio}

\end{document}